\documentclass[reqno]{amsart}
\usepackage{hyperref}

\begin{document}
\title[]
{Normalized solutions to the fractional Kirchhoff equations with combined nonlinearities}

\author[L. Liu, H. Chen, J. Yang\hfil \hfilneg]
{lintao Liu, Haibo Chen, Jie Yang}  
\address{Lintao Liu \newline
School of Mathematics and Statistics, Central South University, Changsha, Hunan 410083, PR China}
\email{liulintao1995@163.com}

\address{Haibo Chen (Corresponding Author)\newline
School of Mathematics and Statistics, Central South University, Changsha, Hunan 410083, PR China}
\email{math\_chb@163.com}

\address{Jie Yang \newline
School of Mathematics and Statistics, Central South University, Changsha, Hunan 410083, PR China}
\email{dafeyang@163.com}

\thanks{}
\keywords{Normalized solutions;    Fractional Kirchhoff equations;   Lagrange multiplier
\hfill\break\indent Decreasing rearrangement.}

\begin{abstract}
In this paper, we study the existence and asymptotic properties of solutions to the following fractional Kirchhoff equation
\begin{equation*}
\left(a+b\int_{\mathbb{R}^{3}}|(-\Delta)^{\frac{s}{2}}u|^{2}dx\right)(-\Delta)^{s}u=\lambda u+\mu|u|^{q-2}u+|u|^{p-2}u \quad \hbox{in $\mathbb{R}^3$,}
\end{equation*}
with a prescribed mass
\begin{equation*}
\int_{\mathbb{R}^{3}}|u|^{2}dx=c^{2},
\end{equation*}
where $s\in(0, 1)$, $a, b, c>0$, $2<q<p<2_{s}^{\ast}=\frac{6}{3-2s}$, $\mu>0$ and $\lambda\in\mathbb{R}$ as a Lagrange multiplier. Under different assumptions on $q<p$, $c>0$ and $\mu>0$, we prove some existence results about the normalized solutions. Our results extend the results of Luo and Zhang (Calc. Var. Partial Differential Equations 59, 1-35, 2020) to the fractional Kirchhoff equations. Moreover, we give some results about the behavior of the normalized solutions obtained above as $\mu\rightarrow0^{+}$.

\end{abstract}

\maketitle
\numberwithin{equation}{section}
\newtheorem{theorem}{Theorem}[section]
\newtheorem{lemma}[theorem]{Lemma}
\newtheorem{definition}[theorem]{Definition}
\newtheorem{remark}[theorem]{Remark}
\newtheorem{proposition}[theorem]{Proposition}
\newtheorem{corollary}[theorem]{Corollary}
\allowdisplaybreaks

\section{Introduction}

In this paper, we are concerned with the existence of solutions $(\lambda, u)\in\mathbb{R}\times H^{s}(\mathbb{R}^{3})$ to the following fractional Kirchhoff equation:
\begin{equation}\label{equ1-1}
\left(a+b\int_{\mathbb{R}^{3}}|(-\Delta)^{\frac{s}{2}}u|^{2}dx\right)(-\Delta)^{s}u=\lambda u+\mu|u|^{q-2}u+|u|^{p-2}u \quad \hbox{in $\mathbb{R}^3$,}
\end{equation}
under the normalized constraint
\begin{equation}\label{equ1-2}
\int_{\mathbb{R}^{3}}|u|^{2}dx=c^{2},
\end{equation}
where $s\in(0, 1)$, $a, b, c>0$, $2<q<p<2_{s}^{\ast}=\frac{6}{3-2s}$ and $\mu>0$ is a parameter. The fractional Laplaction $(-\Delta)^{s}$$(s\in(0, 1))$ can be defined by
\begin{equation*}
(-\Delta)^{s}v(x)=C_{s}P.V.\int_{\mathbb{R}^{3}}\frac{v(x)-v(y)}{|x-y|^{3+2s}}dy=C_{s}\lim_{\varepsilon\rightarrow0}\int_{\mathbb{R}^{3}\backslash B_{\varepsilon}(x)}\frac{v(x)-v(y)}{|x-y|^{3+2s}}dy
\end{equation*}
for $v\in \mathcal{S}(\mathbb{R}^{3})$, where $\mathcal{S}(\mathbb{R}^{3})$ is the Schwartz space of rapidly decaying $\mathcal{C}^{\infty}$ function, $B_{\varepsilon}(x)$ denote an open ball of radius $\varepsilon$ centered at $x$ and the normalization constant $C_{s}=(\int_{\mathbb{R}^{3}}\frac{1-\cos(\zeta_{1})}{|\zeta|^{3+2s}})^{-1}$. For $u\in \mathcal{S}(\mathbb{R}^{3})$, the fractional Laplaction $(-\Delta)^{s}$ can be defined by the Fourier transform $(-\Delta)^{s}u=\mathfrak{F}^{-1}(|\xi|^{2s}\mathfrak{F}u)$, $\mathfrak{F}$ being the usual Fourier transform. The application background of operator $(-\Delta)^{s}$ can be founded in several areas such as fractional quantum mechanics \cite{L,LL}, physics and chemistry \cite{M}, obstacle problems \cite{S}, optimization and finance \cite{C}, conformal geometry and minimal surfaces \cite{CC} and so on. Note that $(-\Delta)^{s}$ on $\mathbb{R}^{N}$$(N\geq1)$ is a nonlocal operator. In the remarkable work of Caffarelli and Silvestre \cite{CCC}, they introduced the extension method that reduced this nonlocal problem into a local one in higher dimensions. More precisely, for a $u\in H^{s}(\mathbb{R}^{N})$, the extension $U: \mathbb{R}^{N}\times [0, \infty)\rightarrow \mathbb{R}$ be defined by
\begin{equation*}
\left\{
\begin{array}{ll}
-div(y^{1-2s}\nabla U)=0, & \hbox{in $\mathbb{R}_{+}^{N+1}$,} \\
U(x, 0)=u, & \hbox{in $\mathbb{R}^{N}$.}
\end{array}
\right.
\end{equation*}
Then, it follows from \cite{CCC} that
\begin{equation*}
(-\Delta)^{s}u(x)=-C_{N, s}\lim_{y\rightarrow0^{+}}y^{1-2s}U_{y}(x, y),
\end{equation*}
where $C_{N, s}$ is a constant depending on $N$ and $s$. This extension method has been successfully applied to nonlinear equations involving a fractional Laplacian, and a series of significant results have been obtained, we refer the readers to see \cite{DD,DDD,LLL,P,SS,WWW} and the references therein.

For the equation \eqref{equ1-1}, a direct choice is to search for solutions $u\in H^{s}(\mathbb{R}^{3})$ by looking for critical points of the functional $F: H^{s}(\mathbb{R}^{3})\rightarrow\mathbb{R}$ defined by:
\begin{equation}\label{equ1-3}
\begin{split}
F(u):&=\frac{a}{2}\int_{\mathbb{R}^{3}}|(-\Delta)^{\frac{s}{2}}u|^{2}dx+\frac{b}{4}\left(\int_{\mathbb{R}^{3}}|(-\Delta)^{\frac{s}{2}}u|^{2}dx \right)^{2}\\
&-\frac{\lambda}{2}\int_{\mathbb{R}^{3}}|u|^{2}dx-\frac{\mu}{q}\int_{\mathbb{R}^{3}}|u|^{q}dx-\frac{1}{p}\int_{\mathbb{R}^{3}}|u|^{p}dx,
\end{split}
\end{equation}
where $\lambda\in\mathbb{R}$ is fixed, please see \cite{HH,LLLLL} for more details.

Moreover, one can look for solutions to \eqref{equ1-1} having a prescribed $L^{2}$-norm, and in this case $\lambda\in\mathbb{R}$ is unknown. It is well known that equation \eqref{equ1-1} arises from looking for the standing wave type solutions $\psi(x, t)=e^{-i\lambda t}u(x)$, $\lambda\in\mathbb{R}$ for the following time-dependent nonlinear fractional Schr\"{o}dinger equation:
\begin{equation}\label{equ1-4}
i\frac{\partial \psi}{\partial t}=(a+b\int_{\mathbb{R}^{3}}|(-\Delta)^{\frac{s}{2}}\psi|^{2}dx)(-\Delta)^{s}\psi-f(|\psi|)\psi, \quad \hbox{in $\mathbb{R}^{3}$,}
\end{equation}
where $0<s<1$, $i$ denotes the imaginary unit and $\psi=\psi(x, t): \mathbb{R}^{3}\times[0, +\infty)\rightarrow\mathbb{C}$. Clearly, $\psi$ solves \eqref{equ1-4} if and only if the standing wave $u(x)$ satisfies \eqref{equ1-1} with $f(u)=u^{p-2}+\mu u^{q-2}$. Let $\psi$ is a solutions of \eqref{equ1-4}, multiplying \eqref{equ1-4} by the conjugate $\bar{\psi}$ of $\psi$, integrating over $\mathbb{R}^{3}$, and taking the imaginary part, we get $\frac{d}{dt}|\psi(t)|_{2}^{2}=0$, which implies that solutions $\psi\in C([0, T); H^{s}(\mathbb{R}^{3}))$ to \eqref{equ1-4} maintain their mass along time, it is natural, this approach seems to be particularly meaningful from the physical point of view.

When $s=1$, $a=1$ and $b=0$ i.e. for the Laplacian operator. In \cite{J}, Jeanjean's firstly proved the existence of normalized solutions in $L^{2}$-supercritical case, after that, there have been a lot of works about the existence and properties of normalized solutions. For instance,  very recent works of Soave \cite{SSSSSS}, author proved the existence and properties of normalized ground states for the NLS equation with combined nonlinearities
\begin{equation*}
-\Delta u=\lambda u+\mu|u|^{q-2}u+|u|^{p-2}u  \quad \hbox{in $\mathbb{R}^N$,}\quad \hbox{$N\geq1$,}
\end{equation*}
under the constraint
\begin{equation*}
\int_{\mathbb{R}^{N}}|u|^{2}dx=a^{2},
\end{equation*}
where $a>0$, $\mu\in\mathbb{R}$, $2<q\leq2+\frac{4}{N}\leq p<2^{\ast}=\frac{2N}{N-2}$ and $q\neq p$.

And in \cite{SSSSSSS}, Soave also proved the existence and properties of normalized ground states for the NLS equation with the Sobolev critical exponent
\begin{equation*}
-\Delta u=\lambda u+\mu|u|^{q-2}u+|u|^{2^{\ast}-2}u  \quad \hbox{in $\mathbb{R}^N$} \quad \hbox{$N\geq3$,}
\end{equation*}
with prescribed mass:
\begin{equation*}
\int_{\mathbb{R}^{N}}|u|^{2}dx=a^{2},
\end{equation*}
where $a>0$, $\mu\in\mathbb{R}$ and $2<q<2^{\ast}=\frac{2N}{N-2}$. For other related and important results, we refer the readers to \cite{GG,I} and their references.

When $s\neq1$, $a=1$ and $b=0$, i.e. for the fractional Schr\"{o}dinger equations, Luo and Zhang \cite{LLLL} proved some existence and nonexistence results about the normalized solutions of the fractional nonlinear Schr\"{o}dinger equations with combined nonlinearities
\begin{equation*}
(-\Delta)^{s}u=\lambda u+\mu|u|^{q-2}u+|u|^{p-2}u  \quad \hbox{in $\mathbb{R}^N$,}
\end{equation*}
with prescribed mass:
\begin{equation*}
\int_{\mathbb{R}^{N}}|u|^{2}dx=a^{2},
\end{equation*}
where $0<s<1$, $N\geq2$, $\mu\in\mathbb{R}$ and $2<q<p<2_{s}^{\ast}=\frac{2N}{N-2s}$.

In \cite{ZZ}, Zhen and Zhang studied the normalized ground states for the following critical fractional NLS equation with prescribed mass
\begin{equation*}
\left\{
\begin{array}{ll}
(-\Delta)^{s}u=\lambda u+\mu|u|^{q-2}u+|u|^{p-2}u \quad \hbox{in $\mathbb{R}^N$,} \\
\int_{\mathbb{R}^{N}}|u|^{2}dx=a^{2},
\end{array}
\right.
\end{equation*}
where $0<s<1$, $N>2s$, $\mu\in\mathbb{R}$ and $2<q<2_{s}^{\ast}=\frac{2N}{N-2s}$.

When $s=1$, $a>0$ and $b>0$, \eqref{equ1-1} reduces to the classical Kirchhoff equation
\begin{equation}\label{equ1-5}
-\left(a+b\int_{\mathbb{R}^{3}}|\nabla u|^{2}dx\right)\Delta u=\lambda u+\mu|u|^{q-2}u+|u|^{p-2}u  \quad \hbox{in $\mathbb{R}^3$.}
\end{equation}
If it is to search for solutions $u\in H^{1}(\mathbb{R}^{3})$ by looking for critical point of the functional $K: H^{1}(\mathbb{R}^{3})\rightarrow\mathbb{R}$ defined by:
\begin{equation*}
\begin{split}
K(u):&=\frac{a}{2}\int_{\mathbb{R}^{3}}|\nabla u|^{2}dx+\frac{b}{4}\left(\int_{\mathbb{R}^{3}}|\nabla u|^{2} dx\right)^{2}\\
&-\frac{\lambda}{2}\int_{\mathbb{R}^{3}}|u|^{2}dx-\frac{\mu}{q}\int_{\mathbb{R}^{3}}|u|^{q}dx-\frac{1}{p}\int_{\mathbb{R}^{3}}|u|^{p}dx,
\end{split}
\end{equation*}
where $\lambda\in\mathbb{R}$ is fixed, then equation \eqref{equ1-5} has been extensively studied, see e.g. \cite{HHH,HHHH} and the references therein. At present, there are only few papers which considered the normalized solutions for the equation \eqref{equ1-5}. For instance, Li, Luo and Yang \cite{LLLLLL} proved the existence and asymptotic properties of solutions to the above Kirchhoff equation \eqref{equ1-5} under the normalized constraint $\int_{\mathbb{R}^{3}}|u|^{2}dx=c^{2}$ with $a, b, c>0$, $2<q<\frac{14}{3}<p\leq6$ or $\frac{14}{3}<q<p\leq6$, $\mu>0$ and $\lambda\in\mathbb{R}$ appears as a Lagrange multiplier. To our best knowledge, the existence of normalized solutions to equation \eqref{equ1-1} with $a, b, c>0$, $\mu>0$ and $2<q<p<2_{s}^{\ast}=\frac{6}{3-2s}$ is still unknown. Inspired by \cite{LLLL}, in this paper, we consider the existence result of normalized solutions for fractional Kirchhoff type problem \eqref{equ1-1} with prescribed mass $\int_{\mathbb{R}^{3}}|u|^{2}dx=c^{2}$.

It is well known that the fractional Sobolev space $H^{s}(\mathbb{R}^{3})$ can be defined as follows
\begin{equation*}
H^{s}(\mathbb{R}^{3})=\{u\in L^{2}(\mathbb{R}^{3})|\int_{\mathbb{R}^{3}}|(-\Delta)^{\frac{s}{2}}u|^{2}dx<+\infty\},
\end{equation*}
endowed with the norm
\begin{equation*}
\|u\|^{2}=\int_{\mathbb{R}^{3}}(|(-\Delta)^{\frac{s}{2}}u|^{2}+|u|^{2})dx,
\end{equation*}
where
\begin{equation*}
\int_{\mathbb{R}^{3}}|(-\Delta)^{\frac{s}{2}}u|^{2}dx=\int\int_{\mathbb{R}^{6}}\frac{|u(x)-u(y)|^{2}}{|x-y|^{3+2s}}dxdy.
\end{equation*}
Moreover, we denote $H^{s}_{r}(\mathbb{R}^{3})$ by
\begin{equation*}
H^{s}_{r}(\mathbb{R}^{3})=\{u\in H^{s}(\mathbb{R}^{3}): u(x)=u(|x|), x\in\mathbb{R}^{3}\}.
\end{equation*}
In order to obtain the normalized solutions, we also define the energy functional $E_{\mu}: H^{s}(\mathbb{R}^{3})\rightarrow\mathbb{R}$ by
\begin{equation}\label{equ1-6}
E_{\mu}(u)=\frac{a}{2}\int_{\mathbb{R}^{3}}|(-\Delta)^{\frac{s}{2}}u|^{2}dx+\frac{b}{4}\left(\int_{\mathbb{R}^{3}}|(-\Delta)^{\frac{s}{2}}u|^{2}dx \right)^{2}-\frac{\mu}{q}\int_{\mathbb{R}^{3}}|u|^{q}dx-\frac{1}{p}\int_{\mathbb{R}^{3}}|u|^{p}dx.
\end{equation}
It is easy to check that $E_{\mu}\in C^{1}(\mathbb{R}^{3}, \mathbb{R})$. Then the weak solutions of \eqref{equ1-1} are corresponding to critical points of the energy functional $E_{\mu}$ under the constraint
\begin{equation*}
M_{c}:=\{u\in H^{s}(\mathbb{R}^{3}): \int_{\mathbb{R}^{3}}|u|^{2}dx=c^{2}\}.
\end{equation*}
In this paper, we will be particularly interested in ground state solutions, so we write that $\tilde{u}$ is a ground state of \eqref{equ1-1} on $M_{c}$ if it satisfies:
\begin{equation*}
dE_{\mu}|_{M_{c}}(\tilde{u})=0 \quad and \quad E_{\mu}(\tilde{u})=\inf_{u\in M_{c}}\{E_{\mu}(u): dE_{\mu}|_{M_{c}}(u)=0\}.
\end{equation*}
For simplicity, we define some important parameters. For $p, q\in(2, 2_{s}^{\ast})$, we denote:
\begin{equation*}
\vartheta_{s, p}=\frac{3(p-2)}{2sp}, \quad \vartheta_{s, q}=\frac{3(q-2)}{2sq}.
\end{equation*}
It is easy to see that $\vartheta_{s, p}, \vartheta_{s, q}\in(0, 1)$ and
\begin{equation*}
\begin{cases}
q\vartheta_{s, q}<2<4<p\vartheta_{s, p}, & 2<q<2+\frac{4s}{3} \quad \hbox{and}\quad 2+\frac{8s}{3}<p<2_{s}^{\ast},\\
4<q\vartheta_{s, q}<p\vartheta_{s, p}, & 2+\frac{8s}{3}<q<p<2_{s}^{\ast}.
\end{cases}
\end{equation*}
For $2<q<2+\frac{4s}{3}$ and $2+\frac{8s}{3}<p<2_{s}^{\ast}$, we also denote:
\begin{equation}\label{equ1-7}
\mu_{1}:=\frac{1}{c^{q(1-\vartheta_{s, q})+\frac{p(1-\vartheta_{s, p})(4-q\vartheta_{s, q})}{p\vartheta_{s, p}-4}}}\left[\frac{q(p\vartheta_{s, p}-4)b}{4(p\vartheta_{s, p}-q\vartheta_{s, q})C^{q}(s, q)}\right]\left[\frac{p(4-q\vartheta_{s, q})b}{4(p\vartheta_{s, p}-q\vartheta_{s, q})C^{p}(s, p)}\right]^{\frac{4-q\vartheta_{s, q}}{p\vartheta_{s, p}-4}},
\end{equation}
\begin{equation}\label{equ1-8}
\mu_{2}:=\frac{qA(p, q, s)}{C^{q}(s, q)}\left[\frac{\frac{a}{2}(\frac{bp}{4C^{p}(s, p)})^{\frac{2-q\vartheta_{s, q}}{p\vartheta_{s, p}-4}}}{c^{q(1-\vartheta_{s, q})+\frac{p(1-\vartheta_{s, p})(2-q\vartheta_{s, q})}{p\vartheta_{s, p}-4}}}+\frac{(\frac{b}{4})^{\frac{p\vartheta_{s, p}-q\vartheta_{s, q}}{p\vartheta_{s, p}-4}}(\frac{p}{C^{p}(s, p)})^{\frac{4-q\vartheta_{s, q}}{p\vartheta_{s, p}-4}}}{c^{q(1-\vartheta_{s, q})+\frac{p(1-\vartheta_{s, p})(4-q\vartheta_{s, q})}{p\vartheta_{s, p}-4}}}\right],
\end{equation}
where $A(p, q, s):=\left(\frac{8(4-q\vartheta_{s, q})}{p\vartheta_{s, p}(p\vartheta_{s, p}-2)(p\vartheta_{s, p}-q\vartheta_{s, q})}\right)^{\frac{4-q\vartheta_{s, q}}{p\vartheta_{s, p}-4}}-\left(\frac{8(4-q\vartheta_{s, q})}{p\vartheta_{s, p}(p\vartheta_{s, p}-2)(p\vartheta_{s, p}-q\vartheta_{s, q})}\right)^{\frac{p\vartheta_{s, p}-q\vartheta_{s, q}}{p\vartheta_{s, p}-4}}$, $C(s, \alpha)$ satisfy the fractional Gagliardo-Nirenberg-Sobolev inequality:
\begin{equation*}
\|u\|_{\alpha}\leq C(s, \alpha)\|(-\Delta)^{\frac{s}{2}}u\|_{2}^{\vartheta_{s, \alpha}}\|u\|_{2}^{1-\vartheta_{s, \alpha}}, \quad \alpha\in(2, 2_{s}^{\ast}),
\end{equation*}
see Section 2 below for details.

Next, we call $2+\frac{8s}{3}$ is the $L^{2}$-critical exponent for \eqref{equ1-1}, since
\begin{equation*}
\begin{cases}
\inf_{u\in M_{c}}E_{\mu}(u)>-\infty, & p, q\in(2, 2+\frac{8s}{3}),\\
\inf_{u\in M_{c}}E_{\mu}(u)=-\infty, & p\in(2+\frac{8s}{3}, 2_{s}^{\ast})\quad \hbox{or}\quad q\in(2+\frac{8s}{3}, 2_{s}^{\ast}).
\end{cases}
\end{equation*}
Therefore, we will consider problem \eqref{equ1-1} in following two cases:\\
$(a)$ the mixed critical case: $a, b, c>0$, $\mu>0$ and $2<q<2+\frac{8s}{3}<p<2_{s}^{\ast}$;\\
$(b)$ the $L^{2}$-supercritical case: $a, b, c>0$, $\mu>0$ and $2+\frac{8s}{3}<q<p<2_{s}^{\ast}$.

Now we state our main result as follows. In the case $(a)$, we have
\begin{theorem}\label{the1-1}
Let $2<q<2+\frac{4s}{3}$, $2+\frac{8s}{3}<p<2_{s}^{\ast}$ and $0<\mu<\min\{\mu_{1}, \mu_{2}\}$. Then\\
$(1)$ $E_{\mu}|_{M_{a}}$ has a ground state $\tilde{u}_{\mu}$ with the following properties: $\tilde{u}_{\mu}$ is a positive, radially symmetric function, and solves \eqref{equ1-1} for some $\tilde{\lambda}_{\mu}<0$. Moreover, denoting by $m(c, \mu)=E_{\mu}(\tilde{u}_{\mu})$, we have $m(c, \mu)<0$ and $\tilde{u}_{\mu}$ is an interior local minimizer of $E_{\mu}$ on the set
\begin{equation*}
A_{k}:=\{u\in M_{a}: \|(-\Delta)^{\frac{s}{2}}u\|_{2}<k\},
\end{equation*}
for some $k>0$ small enough. Any other ground state of $E_{\mu}$ on $M_{a}$ is a local minimizer of $E_{\mu}$ on $A_{k}$.\\
$(2)$ $E_{\mu}$ has a second critical point of Mountain Pass type $\hat{u}_{\mu}$ at some energy level $\sigma(c, \mu)>0$ and $\hat{u}_{\mu}$ with the following properties: $\hat{u}_{\mu}$ is a positive, radially symmetric function, and solves \eqref{equ1-1} for some $\hat{\lambda}_{\mu}<0$.

\end{theorem}

Moreover, we also are concerned with behavior of the ground states in Theorem \ref{the1-1} as $\mu\rightarrow0^{+}$.
\begin{theorem}\label{the1-2}
Under the assumptions of Theorem \ref{the1-1}. Then\\
$(1)$ $m(c, \mu)\rightarrow0$ and $\|(-\Delta)^{\frac{s}{2}}\tilde{u}_{\mu}\|_{2}\rightarrow0$ as $\mu\rightarrow0^{+}$;\\
$(2)$ $\sigma(c, \mu)\rightarrow m(c, 0)$ and $\hat{u}_{\mu}\rightarrow u_{0}$ in $H^{s}(\mathbb{R}^{3})$ as $\mu\rightarrow0^{+}$, where $m(c, 0)=E_{\mu}(u_{0})$ and $u_{0}$ is the unique ground state of $E_{0}|_{M_{c}}$.
\end{theorem}

In the case $(b)$, we have the following results.

\begin{theorem}\label{the1-3}
Let $2+\frac{8s}{3}<q<p<2_{s}^{\ast}$ and $\mu>0$. Then $E_{\mu}$ has a critical point of Mountain Pass type $\hat{u}_{\mu}$ at some energy level $\sigma(c, \mu)>0$ with the following properties: $\hat{u}_{\mu}$ is a positive ground state, radially symmetric function, and solves \eqref{equ1-1} for some $\hat{\lambda}_{\mu}<0$.
\end{theorem}

In addition, we also are concerned with the behavior of the ground states in Theorem \ref{the1-3} as $\mu\rightarrow0^{+}$.
\begin{theorem}\label{the1-4}
Under the assumptions of Theorem \ref{the1-3}. Then $\sigma(c, \mu)\rightarrow m(c, 0)$ and $\hat{u}_{\mu}\rightarrow u_{0}$ in $H^{s}(\mathbb{R}^{3})$ as $\mu\rightarrow0^{+}$, where $m(c, 0)=E_{\mu}(u_{0})$ and $u_{0}$ is the unique ground state of $E_{0}|_{M_{c}}$.
\end{theorem}

\begin{remark}\label{the1-5}
We should point out that Luo and Zhang in \cite{LLLL} considered the fractional Schr\"{o}dinger equation with combined nonlinearities and proved the existence and nonexistence of normalized solutions. Compared with the fractional Schr\"{o}dinger equation, the fractional Kirchhoff equation is more difficult since the corresponding fiber map $\mathcal{J}_{u}^{\mu}(\tau)$ has four different terms (see \eqref{equ1-11} below).
\end{remark}
\begin{remark}\label{the1-6}
Very recently, Li, Luo and Yang \cite{LLLLLL} proved the existence and asymptotic properties of solutions to the Kirchhoff equation. Compared to the Laplacian problems, the fractional Laplacian problems are nonlocal and more challenging. In fact, when we consider the fractional Laplacian problem, the corresponding algebraic equation is about fractional order, which is more complicated to deal with than an integer-order algebraic equation.
\end{remark}

To prove our main results, a Pohozaev set needs to be defined by
\begin{equation}\label{equ1-9}
\mathcal{P}_{c, \mu}=\{u\in M_{c}: P_{\mu}(u)=0\},
\end{equation}
where
\begin{equation}\label{equ1-10}
P_{\mu}(u)=a\|(-\Delta)^{\frac{s}{2}}u\|_{2}^{2}+b\|(-\Delta)^{\frac{s}{2}}u\|_{2}^{4}-\mu\vartheta_{s, q}\|u\|_{q}^{q}-\vartheta_{s, p}\|u\|_{p}^{p}.
\end{equation}
Clearly, Lemma \ref{lem2-3} implies that any critical points of $E_{\mu}|_{M_{c}}$ stay in $\mathcal{P}_{c, \mu}$. Moreover, for $u\in M_{c}$ and $\tau\in\mathbb{R}$, let
\begin{equation*}
(\tau\star u)(x)=e^{\frac{3}{2}\tau}u(e^{\tau}x), \quad \hbox{ a.e. in $\mathbb{R}^{3}$,}
\end{equation*}
then $\tau\star u\in M_{c}$. Hence we introduce the fiber map
\begin{equation}\label{equ1-11}
\begin{split}
\mathcal{J}_{u}^{\mu}(\tau):&=E_{\mu}(\tau\star u)
=\frac{ae^{2s\tau}}{2}\|(-\Delta)^{\frac{s}{2}}u\|_{2}^{2}+\frac{be^{4s\tau}}{4}\|(-\Delta)^{\frac{s}{2}}u\|_{2}^{4}\\
&-\mu\frac{e^{3\tau(\frac{q}{2}-1)}}{q}\|u\|_{q}^{q}-\frac{e^{3\tau(\frac{p}{2}-1)}}{p}\|u\|_{p}^{p}.
\end{split}
\end{equation}

It is easy to check that $(\mathcal{J}_{u}^{\mu})'(\tau)=sP_{\mu}(\tau\star u)$, $0<s<1$. So that $\tau$ is a critical point of $\mathcal{J}_{u}^{\mu}$ if and only if $\tau\star u\in \mathcal{P}_{c, \mu}$. Clearly, $u\in \mathcal{P}_{c, \mu}$ if and only if $0$ is a critical point of $\mathcal{J}_{u}^{\mu}$. Thus, we consider to decomposite $\mathcal{P}_{c, \mu}$ into the disjoint union $\mathcal{P}_{c, \mu}=\mathcal{P}_{c, \mu}^{+}\bigcup\mathcal{P}_{c, \mu}^{0}\bigcup\mathcal{P}_{c, \mu}^{-}$, where
\begin{equation}\label{equ1-12}
\begin{split}
\mathcal{P}_{c, \mu}^{+}&=\{u\in \mathcal{P}_{c, \mu}: (\mathcal{J}_{u}^{\mu})''(0)>0\}\\
&=\{u\in \mathcal{P}_{c, \mu}: 2a\|(-\Delta)^{\frac{s}{2}}u\|_{2}^{2}+4b\|(-\Delta)^{\frac{s}{2}}u\|_{2}^{4}>\mu q\vartheta^{2}_{s, q}\|u\|_{q}^{q}+p\vartheta^{2}_{s, p}\|u\|_{p}^{p}\};
\end{split}
\end{equation}
\begin{equation}\label{equ1-13}
\begin{split}
\mathcal{P}_{c, \mu}^{0}&=\{u\in \mathcal{P}_{c, \mu}: (\mathcal{J}_{u}^{\mu})''(0)=0\}\\
&=\{u\in \mathcal{P}_{c, \mu}: 2a\|(-\Delta)^{\frac{s}{2}}u\|_{2}^{2}+4b\|(-\Delta)^{\frac{s}{2}}u\|_{2}^{4}=\mu q\vartheta^{2}_{s, q}\|u\|_{q}^{q}+p\vartheta^{2}_{s, p}\|u\|_{p}^{p}\};
\end{split}
\end{equation}
\begin{equation}\label{equ1-14}
\begin{split}
\mathcal{P}_{c, \mu}^{-}&=\{u\in \mathcal{P}_{c, \mu}: (\mathcal{J}_{u}^{\mu})''(0)<0\}\\
&=\{u\in \mathcal{P}_{c, \mu}: 2a\|(-\Delta)^{\frac{s}{2}}u\|_{2}^{2}+4b\|(-\Delta)^{\frac{s}{2}}u\|_{2}^{4}<\mu q\vartheta^{2}_{s, q}\|u\|_{q}^{q}+p\vartheta^{2}_{s, p}\|u\|_{p}^{p}\}.
\end{split}
\end{equation}

In the sequel, we use the following notations:
{\small

\begin{itemize}

\item
$C$ and $C_{i}$ denotes a universal positive constant.

\item
For $1\leq p<\infty$ and $u\in L^{p}(\mathbb{R}^{3})$, we denote $\|u\|_{p}=(\int_{\mathbb{R}^{3}}|u|^{p}dx)^{\frac{1}{p}}$.

\item
$M_{c, r}:=H^{s}_{r}\cap M_{c}=\{u\in H^{s}_{r}(\mathbb{R}^{3}): \|u\|_{2}^{2}=c^{2}\}$.

\item
Set
\begin{equation*}
D^{s,2}(\mathbb{R}^{3}):=\{u\in L^{2^{\ast}_{s}}(\mathbb{R}^{3}):(-\Delta)^{\frac{s}{2}}u\in L^{2}(\mathbb{R}^{3})\}
\end{equation*}
denoting the homogeneous fractional sobolev space with the norm
\begin{equation*}
\|u\|^{2}_{D^{s,2}}:=\int_{\mathbb{R}^{3}}|(-\Delta)^{\frac{s}{2}}u|^{2}dx.
\end{equation*}

\item
Denote by $^{\ast}$ the symmetric decreasing rearrangement of a $H^{s}(\mathbb{R}^{3})$ function.

\item
The symbol $\rightharpoonup$ denotes weak convergence and the symbol $\rightarrow$ denotes strong convergence.

\item
The symbol $\langle\cdot,\cdot \rangle$ denotes the dual pair for any Banach space and its dual space.

\end{itemize}

The paper is organized as follows. In Section 2, we present some preliminaries results. In Section 3, we establish a compactness lemma. Section 4 is devoted to the mixed critical case and we will prove Theorem \ref{the1-1} and Theorem \ref{the1-2}. In Section 5, we study the purely $L^{2}$-supercritical case and prove Theorem \ref{the1-3} and Theorem \ref{the1-4}.

\section{preliminary lemmas}
In this section, we give some lemmas which will be often used throughout the rest of the paper. First, we give the fractional Gagliardo-Nirenberg inequality.

\begin{lemma}\label{lem2-1}\cite{LLLL}
If $\alpha\in(2, 2_{s}^{\ast})$, there exists an optimal constant $C(s, \alpha)$ such that
\begin{equation}\label{equ2-1}
\|u\|_{\alpha}\leq C(s, \alpha)\|(-\Delta)^{\frac{s}{2}}u\|_{2}^{\vartheta_{s, \alpha}}\|u\|_{2}^{1-\vartheta_{s, \alpha}} \quad \hbox{  $\forall u\in H^{s}(\mathbb{R}^{3})$,}
\end{equation}
where $\vartheta_{s, \alpha}$ is defined by $\vartheta_{s, \alpha}=\frac{3(\alpha-2)}{2s\alpha}$.

\end{lemma}

Now, we recall the Pohozaev identity for fractional scalar field equations.
\begin{lemma}\label{lem2-2}\cite{LLLL}
Let $a, b>0$ and $u\in H^{s}(\mathbb{R}^{3})$ satisfies the equation
\begin{equation*}
\left(a+b\int_{\mathbb{R}^{3}}|(-\Delta)^{\frac{s}{2}}u|^{2}dx\right)(-\Delta)^{s}u=g(u),
\end{equation*}
then it holds that
\begin{equation*}
a\frac{3-2s}{2}\int_{\mathbb{R}^{3}}|(-\Delta)^{\frac{s}{2}}u|^{2}dx+b\frac{3-2s}{2}\left(\int_{\mathbb{R}^{3}}|(-\Delta)^{\frac{s}{2}}u|^{2}dx\right)^{2}=3\int_{\mathbb{R}^{3}}G(u)dx,
\end{equation*}
where $G(s)=\int_{0}^{s}g(t)dt$.
\end{lemma}

\begin{lemma}\label{lem2-3}\cite{LLLL}
Let $s\in(0, 1)$. For any $u\in H^{s}(\mathbb{R}^{3})$, the following inequality holds
\begin{equation*}
\int\int_{\mathbb{R}^{6}}\frac{(u^{\ast}(x)-u^{\ast}(y))^{2}}{|x-y|^{3+2s}}dxdy\leq \int\int_{\mathbb{R}^{6}}\frac{(u(x)-u(y))^{2}}{|x-y|^{3+2s}}dxdy.
\end{equation*}
\end{lemma}

\begin{lemma}\label{lem2-4}\cite{LLLL}
Let $N\geq2$, then $H_{r}^{s}(\mathbb{R}^{N})$ is compactly embedding into $L^{p}(\mathbb{R}^{N})$ for $p\in(2, 2_{s}^{\ast})$.
\end{lemma}
For $u\in M_{c}$ and $\tau\in\mathbb{R}$, by a simple calculation, we have
\begin{equation*}
\begin{split}
&(\mathcal{J}_{u}^{\mu})'(\tau)=sae^{2s\tau}\|(-\Delta)^{\frac{s}{2}}u\|^{2}_{2}+sbe^{4s\tau}\|(-\Delta)^{\frac{s}{2}}u\|^{4}_{2}-\mu\vartheta_{s, q}se^{q\vartheta_{s, q}s\tau}\|u\|^{q}_{q}-\vartheta_{s, p}se^{p\vartheta_{s, p}s\tau}\|u\|^{p}_{p}\\
&=sa\|(-\Delta)^{\frac{s}{2}}(\tau\star u)\|^{2}_{2}+sb\|(-\Delta)^{\frac{s}{2}}(\tau\star u)\|^{4}_{2}-s\mu\vartheta_{s, q}\|\tau\star u\|^{q}_{q}-s\vartheta_{s, p}\|\tau\star u\|^{p}_{p}=sP_{\mu}(\tau\star u),
\end{split}
\end{equation*}
where $P_{\mu}$ is defined by \eqref{equ1-10}. Hence, we can see that
\begin{lemma}\label{lem2-5}
Let $u\in M_{c}$. Then $\tau\in\mathbb{R}$ is a critical point of $\mathcal{J}_{u}^{\mu}$ if and only if $\tau\star u\in \mathcal{P}_{c, \mu}$.
\end{lemma}

Next, we recall a version of linking theorem\cite{LLLL}.

\begin{definition}\label{lem2-6}
Let $B$ be a closed subset of $X$. We shall say that a class $\mathcal{F}$ of compact subsets of $X$ is homotopy-stable family with extended boundary $B$ if for any set $A$ in $\mathcal{F}$ and any $\eta\in C([0, 1]\times X; X)$ satisfying $\eta(t, x)=x$ for all $(t, x)\in(\{0\}\times X)\cup([0, 1]\times B)$ we have that $\eta(\{1\}\times A)\in\mathcal{F}$.
\end{definition}

\begin{lemma}\label{lem2-7}
Let $\varphi$ be a $C^{1}$-functional on a complete connected $C^{1}$-Finsler manifold $X$ and consider a homotopy-stable family $\mathcal{F}$ with extended boundary $B$. Set
\begin{equation*}
c=c(\varphi, \mathcal{F})=\inf_{A\in\mathcal{F}}\max_{x\in A}\varphi(x),
\end{equation*}
and let $F$ be a closed subset of $X$ satisfying
\begin{equation}\label{equ2-2}
A\cap F\backslash B\neq\emptyset, \quad \hbox{ for $A\in\mathcal{F}$},
\end{equation}
and
\begin{equation}\label{equ2-3}
\sup_{x\in B}\varphi(x)\leq c\leq\inf_{x\in F}\varphi(x).
\end{equation}
Then, for any sequence of sets $(A_{n})_{n}$ in $\mathcal{F}$ such that $\lim_{n}\sup_{A_{n}}\varphi=c$, there exists a sequence $(x_{n})_{n}$ in $X\backslash B$ such that\\
$(i)$ $\lim_{n}\varphi(x_{n})=c$;\\
$(ii)$ $\lim_{n}\|d\varphi(x_{n})\|=0$;\\
$(iii)$ $\lim_{n}dist(x_{n}, F)=0$;\\
$(iv)$ $\lim_{n}dist(x_{n}, A_{n})=0$.
\end{lemma}
Finally, we give the following two technical results.
\begin{lemma}\label{lem2-8}\cite{LLLLLL}
Let $\hat{a}, \hat{b}, \hat{c}, \hat{d}, \hat{p}, \hat{q}\in(0, +\infty)$ and $f(t)=\hat{a}t^{2}+\hat{b}t^{4}-\hat{c}t^{\hat{p}}-\hat{d}t^{\hat{q}}$ for $t\geq0$. If $\hat{p}\in(4, +\infty)$, $\hat{q}\in(0, 2)$ and $\left[(\frac{8(4-\hat{q})}{\hat{p}(\hat{p}-2)(\hat{p}-\hat{q})})^{\frac{4-\hat{q}}{\hat{p}-4}}-(\frac{8(4-\hat{q})}{\hat{p}(\hat{p}-2)(\hat{p}-\hat{q})})^{\frac{\hat{p}-\hat{q}}{\hat{p}-4}}\right]\left[\frac{\hat{a}}{\hat{d}}(\frac{\hat{b}}{\hat{c}})^{\frac{2-\hat{q}}{\hat{p}-4}}+\frac{1}{\hat{d}}\frac{\hat{b}^{\frac{\hat{p}-\hat{q}}{\hat{p}-4}}}{\hat{c}^{\frac{4-\hat{q}}{\hat{p}-4}}}\right]>1$, then $f(t)$ has a local strict minimum at a negative level and a global strict maximum at a positive level on $[0, +\infty)$.
\end{lemma}
\begin{lemma}\label{lem2-9}\cite{LLLLLL}
Let $\hat{a}, \hat{b}, \hat{c}, \hat{d}, \hat{p}, \hat{q}\in(0, +\infty)$ and $f(t)=\hat{a}t^{2}+\hat{b}t^{4}-\hat{c}t^{\hat{p}}-\hat{d}t^{\hat{q}}$ for $t\geq0$. If $\hat{p}, \hat{q}\in(4, +\infty)$, then $f(t)$ has a unique maximum point at a positive level on $[0, +\infty)$.
\end{lemma}

\section{compactness result of Palais-Smale sequences}

In this section, we give a compact result about Palais-Smale sequences (PS sequences for short), which is crucial to prove our main result.

\begin{proposition}\label{lem3-1}
Let $2<q<2+\frac{8s}{3}<p<2_{s}^{\ast}$ or $2+\frac{8s}{3}<q<p<2_{s}^{\ast}$. Let $\{u_{n}\}\subset M_{c, r}$ be a Palais-Smale sequence for $E_{\mu}|_{M_{c}}$ at level $m\neq0$ with $P_{\mu}(u_{n})\rightarrow0$ as $n\rightarrow\infty$. Then, up to a subsequence, $u_{n}\rightarrow u$ in $H^{s}(\mathbb{R}^{3})$, $E_{\mu}(u)=m$ and $u$ is a radial solution to \eqref{equ1-1} for some $\lambda<0$.
\end{proposition}

\begin{proof}
Firstly, we need to show that $\{u_{n}\}$ is bounded in $H^{s}(\mathbb{R}^{3})$. Indeed, if $2<q<2+\frac{8s}{3}<p<2_{s}^{\ast}$, then $q\vartheta_{s, q}<4<p\vartheta_{s, p}$. By $P_{\mu}(u_{n})\rightarrow0$, we have that
\begin{equation*}
\begin{split}
E_{\mu}(u_{n})&=(\frac{a}{2}-\frac{a}{p\vartheta_{s, p}})\|(-\Delta)^{\frac{s}{2}}u_{n}\|_{2}^{2}+(\frac{b}{4}-\frac{b}{p\vartheta_{s, p}})\|(-\Delta)^{\frac{s}{2}}u_{n}\|_{2}^{4}\\
&-\frac{\mu}{q}(1-\frac{q\vartheta_{s, q}}{p\vartheta_{s, p}})\|u_{n}\|_{q}^{q}+o_{n}(1).
\end{split}
\end{equation*}
Using $E_{\mu}(u_{n})\leq m+1$ for $n$ large and Lemma \ref{lem2-1}, we can see that
\begin{equation*}
\begin{split}
&(\frac{a}{2}-\frac{a}{p\vartheta_{s, p}})\|(-\Delta)^{\frac{s}{2}}u_{n}\|_{2}^{2}+(\frac{b}{4}-\frac{b}{p\vartheta_{s, p}})\|(-\Delta)^{\frac{s}{2}}u_{n}\|_{2}^{4}\\
&\leq m+1+\frac{\mu}{q}(1-\frac{q\vartheta_{s, q}}{p\vartheta_{s, p}})C^{q}(s, q)\|(-\Delta)^{\frac{s}{2}}u\|_{2}^{q\vartheta_{s, q}}\|u\|_{2}^{q(1-\vartheta_{s, q})},
\end{split}
\end{equation*}
which implies that $\{u_{n}\}$ is bounded in $H^{s}(\mathbb{R}^{3})$.

If $2+\frac{8s}{3}<q<p<2_{s}^{\ast}$, then $4<q\vartheta_{s, q}<p\vartheta_{s, p}$. By using $P_{\mu}(u_{n})\rightarrow0$, we get that
\begin{equation*}
\begin{split}
E_{\mu}(u_{n})&=\frac{a}{4}\|(-\Delta)^{\frac{s}{2}}u_{n}\|_{2}^{2}+\mu(\frac{\vartheta_{s, q}}{4}-\frac{1}{q})\|u_{n}\|_{q}^{q}+(\frac{\vartheta_{s, p}}{4}-\frac{1}{p})\|u_{n}\|_{p}^{p}+o_{n}(1)\\
&=\frac{a}{4}\|(-\Delta)^{\frac{s}{2}}u_{n}\|_{2}^{2}+\mu\frac{q\vartheta_{s, q}-4}{4q}\|u_{n}\|_{q}^{q}+\frac{p\vartheta_{s, p}-4}{4p}\|u_{n}\|_{p}^{p}+o_{n}(1)\\
&\leq m+1,
\end{split}
\end{equation*}
for $n$ large, which implies that $\{u_{n}\}$ is bounded in $H^{s}(\mathbb{R}^{3})$.

Now, since $\{u_{n}\}$ is a bounded sequence of radial functions, using Lemma \ref{lem2-4}, up to a subsequence, we may assume that $u_{n}\rightharpoonup u$ in $H^{s}(\mathbb{R}^{3})$, $u_{n}\rightarrow u$ in $L^{r}(\mathbb{R}^{3})$, for any $r\in(2, 2_{s}^{\ast})$ and $u_{n}\rightarrow u$ a.e. in $\mathbb{R}^{3}$. By applying $\{u_{n}\}$ is a bounded PS sequence for $E_{\mu}|_{M_{c}}$ and the Lagrange multipliers rule, there exists $\{\lambda_{n}\}\subset \mathbb{R}$ such that
\begin{equation}\label{equ3-1}
\begin{split}
&(a+b\|(-\Delta)^{\frac{s}{2}}u_{n}\|_{2}^{2})\int_{\mathbb{R}^{3}}(-\Delta)^{\frac{s}{2}}u_{n}(-\Delta)^{\frac{s}{2}}\varphi dx-\lambda_{n}\int_{\mathbb{R}^{3}}u_{n}\varphi dx\\
&-\mu\int_{\mathbb{R}^{3}}|u_{n}|^{q-2}u_{n}\varphi dx-\int_{\mathbb{R}^{3}}|u_{n}|^{p-2}u_{n}\varphi dx=o(1)\|\varphi\|,\quad \hbox{  $\varphi\in H^{s}(\mathbb{R}^{3})$},
\end{split}
\end{equation}
as $n\rightarrow\infty$. Letting $\varphi=u_{n}$, from the boundedness of $\{u_{n}\}$ in $H^{s}(\mathbb{R}^{3})\cap L^{p}(\mathbb{R}^{3})\cap L^{q}(\mathbb{R}^{3})$, we get that $\{\lambda_{n}\}$ is bounded, and hence up to a subsequence, we may assume that $\lambda_{n}\rightarrow\lambda\in\mathbb{R}$. By using $P_{\mu}(u_{n})\rightarrow0$ and $\vartheta_{s, q}<1$, we can see that
\begin{equation*}
\begin{split}
\lambda c^{2}&=\lim_{n\rightarrow\infty}\lambda_{n}\|u_{n}\|_{2}^{2}=\lim_{n\rightarrow\infty}\left(\mu(\vartheta_{s, q}-1)\|u_{n}\|_{q}^{q}+(\vartheta_{s, p}-1)\|u_{n}\|_{p}^{p}+o_{n}(1)\right)\\
&=\mu(\vartheta_{s, q}-1)\|u\|_{q}^{q}+(\vartheta_{s, p}-1)\|u\|_{p}^{p}\leq0.
\end{split}
\end{equation*}
Now, we claim that $u$ does not vanish. Assume by contradiction that $u\equiv0$, we have $\lim_{n\rightarrow\infty}\|u_{n}\|_{q}^{q}=\lim_{n\rightarrow\infty}\|u_{n}\|_{p}^{p}=0$. Using $P_{\mu}(u_{n})\rightarrow0$, we can see that $E_{\mu}(u_{n})\rightarrow0$, which is a contradiction with $E_{\mu}(u_{n})\rightarrow m\neq0$. Thus $\lambda_{n}\rightarrow\lambda<0$ and $u\not\equiv 0$.

Next, we show that $u_{n}\rightarrow u$ in $H^{s}(\mathbb{R}^{3})$. Indeed, since $u_{n}\rightharpoonup u\not\equiv 0$ in $H^{s}(\mathbb{R}^{3})$, we can see that $B:=\lim_{n\rightarrow\infty}\|(-\Delta)^{\frac{s}{2}}u_{n}\|_{2}^{2}\geq\|(-\Delta)^{\frac{s}{2}}u\|_{2}^{2}>0$. Thus, from \eqref{equ3-1}, we get
\begin{equation}\label{equ3-2}
\begin{split}
&(a+bB)\int_{\mathbb{R}^{3}}(-\Delta)^{\frac{s}{2}}u(-\Delta)^{\frac{s}{2}}\varphi dx-\lambda\int_{\mathbb{R}^{3}}u\varphi dx\\
&-\mu\int_{\mathbb{R}^{3}}|u|^{q-2}u\varphi dx-\int_{\mathbb{R}^{3}}|u|^{p-2}u\varphi dx=0,\quad \hbox{  $\varphi\in H^{s}(\mathbb{R}^{3})$}.
\end{split}
\end{equation}
Test \eqref{equ3-1}-\eqref{equ3-2} with $\varphi=u_{n}-u$, we can see that
\begin{equation*}
(a+bB)\int_{\mathbb{R}^{3}}|(-\Delta)^{\frac{s}{2}}(u_{n}-u)|^{2}dx-\lambda\int_{\mathbb{R}^{3}}|u_{n}-u|^{2}dx\rightarrow0.
\end{equation*}

\end{proof}

\section{mixed critical case}
In this section, we always assume that $2<q<2+\frac{4s}{3}$ and $2+\frac{8s}{3}<p<2_{s}^{\ast}$, then $q\vartheta_{s, q}<2$ and $p\vartheta_{s, p}>4$. We will show Theorems \ref{the1-1}-\ref{the1-2}. From $\mathcal{P}_{c, \mu}=\mathcal{P}_{c, \mu}^{+}\bigcup\mathcal{P}_{c, \mu}^{0}\bigcup\mathcal{P}_{c, \mu}^{-}$, we have next result.

\begin{lemma}\label{lem4-1}
Let $2<q<2+\frac{4s}{3}$, $2+\frac{8s}{3}<p<2_{s}^{\ast}$ and $0<\mu<\mu_{1}$. Then $\mathcal{P}_{c, \mu}^{0}=\emptyset$ and $\mathcal{P}_{c, \mu}$ is a smooth manifold of codimension $2$ in $M_{c}$. Here $\mu_{1}$ was defined in \eqref{equ1-7}.
\end{lemma}
\begin{proof}
If $u\in\mathcal{P}_{c, \mu}^{0}$, from $P_{\mu}(u)=0$, we have
\begin{equation*}
a\|(-\Delta)^{\frac{s}{2}}u\|_{2}^{2}+b\|(-\Delta)^{\frac{s}{2}}u\|_{2}^{4}=\mu\vartheta_{s, q}\|u\|_{q}^{q}+\vartheta_{s, p}\|u\|_{p}^{p},
\end{equation*}
and
\begin{equation*}
2a\|(-\Delta)^{\frac{s}{2}}u\|_{2}^{2}+4b\|(-\Delta)^{\frac{s}{2}}u\|_{2}^{4}=\mu q\vartheta^{2}_{s, q}\|u\|_{q}^{q}+p\vartheta^{2}_{s, p}\|u\|_{p}^{p}.
\end{equation*}
From Lemma \ref{lem2-1}, we have
\begin{equation}\label{equ4-1}
\begin{split}
&(2-q\vartheta_{s, q})a\|(-\Delta)^{\frac{s}{2}}u\|_{2}^{2}+(4-q\vartheta_{s, q})b\|(-\Delta)^{\frac{s}{2}}u\|_{2}^{4}\\
&\leq\vartheta_{s, p}(p\vartheta_{s, p}-q\vartheta_{s, q})C^{p}(s, p)\|(-\Delta)^{\frac{s}{2}}u\|_{2}^{p\vartheta_{s, p}}c^{(1-\vartheta_{s, p})p},
\end{split}
\end{equation}
and
\begin{equation}\label{equ4-2}
\begin{split}
&(p\vartheta_{s, p}-2)a\|(-\Delta)^{\frac{s}{2}}u\|_{2}^{2}+(p\vartheta_{s, p}-4)b\|(-\Delta)^{\frac{s}{2}}u\|_{2}^{4}\\
&\leq\mu\vartheta_{s, q}(p\vartheta_{s, p}-q\vartheta_{s, q})C^{q}(s, q)\|(-\Delta)^{\frac{s}{2}}u\|_{2}^{q\vartheta_{s, q}}c^{(1-\vartheta_{s, q})q}.
\end{split}
\end{equation}
By \eqref{equ4-1}-\eqref{equ4-2}, we can get
\begin{equation*}
\mu\geq\frac{1}{c^{q(1-\vartheta_{s, q})+\frac{p(1-\vartheta_{s, p})(4-q\vartheta_{s, q})}{p\vartheta_{s, p}-4}}}\left[\frac{(p\vartheta_{s, p}-4)b}{\vartheta_{s, q}(p\vartheta_{s, p}-q\vartheta_{s, q})C^{q}(s, q)}\right]\left[\frac{(4-q\vartheta_{s, q})b}{\vartheta_{s, p}(p\vartheta_{s, p}-q\vartheta_{s, q})C^{p}(s, p)}\right]^{\frac{4-q\vartheta_{s, q}}{p\vartheta_{s, p}-4}},
\end{equation*}
which implies that $\mu>\mu_{1}$ by the fact that $(\frac{q\vartheta_{s, q}}{4})^{p\vartheta_{s, p}-4}(\frac{p\vartheta_{s, p}}{4})^{4-q\vartheta_{s, q}}<1$. This is a contradiction with $0<\mu<\mu_{1}$. Therefore, $\mathcal{P}_{c, \mu}^{0}=\emptyset$. The rest of the proof is very similar to the Lemma 5.2 in \cite{SSSSSS}, we omit it.

\end{proof}

\begin{lemma}\label{lem4-2}
Let $2<q<2+\frac{4s}{3}$, $2+\frac{8s}{3}<p<2_{s}^{\ast}$ and $0<\mu<\mu_{1}$. If $u\in \mathcal{P}_{c, \mu}$ is a critical point for $E_{\mu}|_{\mathcal{P}_{c, \mu}}$, then $u$ is a critical point for $E_{\mu}|_{M_{c}}$. Here $\mu_{1}$ was defined in \eqref{equ1-7}.
\end{lemma}
\begin{proof}
If $u\in \mathcal{P}_{c, \mu}$ is a critical point for $E_{\mu}|_{\mathcal{P}_{c, \mu}}$, by the Lagrange multipliers rule, we know that there exist $\lambda, \nu\in\mathbb{R}$ such that
\begin{equation*}
\langle E'_{\mu}(u), \varphi\rangle-\lambda\int_{\mathbb{R}^{3}}u\varphi dx-\langle P'_{\mu}(u), \varphi\rangle=0,\quad \hbox{  $\forall\varphi\in H^{s}(\mathbb{R}^{3})$,}
\end{equation*}
that is
\begin{equation}\label{equ4-3}
\begin{split}
&\left[(1-2\nu)a+(1-4\nu)b\int_{\mathbb{R}^{3}}|(-\Delta)^{\frac{s}{2}}u|^{2}dx\right](-\Delta)^{s}u\\
&=\lambda u+\mu(1-\nu q\vartheta_{s, q})|u|^{q-2}u+(1-\nu p\vartheta_{s, p})|u|^{p-2}u  \quad \hbox{in $\mathbb{R}^3$.}
\end{split}
\end{equation}
By Lemma \ref{lem2-2}, we can get that
\begin{equation}\label{equ4-4}
\begin{split}
&\frac{3-2s}{2}(1-2\nu)a\int_{\mathbb{R}^{3}}|(-\Delta)^{\frac{s}{2}}u|^{2}dx+\frac{3-2s}{2}(1-4\nu)b\left(\int_{\mathbb{R}^{3}}|(-\Delta)^{\frac{s}{2}}u|^{2}dx\right)^{2}\\
&-\frac{3}{2}\lambda \int_{\mathbb{R}^{3}}u^{2}dx+\frac{3\mu(\nu q\vartheta_{s, q}-1)}{q}\int_{\mathbb{R}^{3}}|u|^{q}dx+\frac{3(\nu p\vartheta_{s, p}-1)}{q}\int_{\mathbb{R}^{3}}|u|^{p}dx=0 \quad \hbox{in $\mathbb{R}^3$.}
\end{split}
\end{equation}
By \eqref{equ4-3} and \eqref{equ4-4}, we have
\begin{equation}\label{equ4-5}
\begin{split}
&(1-2\nu)a\|(-\Delta)^{\frac{s}{2}}u\|_{2}^{2}+(1-4\nu)b\|(-\Delta)^{\frac{s}{2}}u\|_{2}^{4}\\
&+\mu\vartheta_{s, q}(\nu q\vartheta_{s, q}-1)\|u\|_{q}^{q}+\vartheta_{s, p}(\nu p\vartheta_{s, p}-1)\|u\|_{p}^{p}=0.
\end{split}
\end{equation}
By \eqref{equ4-5} and $P_{\mu}(u)=0$, we can get that
\begin{equation*}
\nu(2a\|(-\Delta)^{\frac{s}{2}}u\|_{2}^{2}+4b\|(-\Delta)^{\frac{s}{2}}u\|_{2}^{4}-\mu q\vartheta^{2}_{s, q}\|u\|_{q}^{q}-p\vartheta^{2}_{s, q}\|u\|_{p}^{p})=0,
\end{equation*}
which implies that $\nu=0$ since $u\not\in\mathcal{P}_{c, \mu}^{0}$.

\end{proof}

Note that by Lemma \ref{lem2-1}, we can get that
\begin{equation}\label{equ4-6}
\begin{split}
E_{\mu}(u)&\geq\frac{a}{2}\|(-\Delta)^{\frac{s}{2}}u\|_{2}^{2}+\frac{b}{4}\|(-\Delta)^{\frac{s}{2}}u\|_{2}^{4}
-\frac{\mu}{q}C^{q}(s, q)\|(-\Delta)^{\frac{s}{2}}u\|_{2}^{q\vartheta_{s, q}}c^{q(1-\vartheta_{s, q})}\\
&-\frac{1}{p}C^{p}(s, p)\|(-\Delta)^{\frac{s}{2}}u\|_{2}^{p\vartheta_{s, p}}c^{p(1-\vartheta_{s, p})}, \quad \hbox{  $\forall u\in M_{c}$,}
\end{split}
\end{equation}
for any $u\in S_{a}$. Therefore, in order to understand the geometry of the functional $E_{\mu}|_{M_{c}}$, we need to set the function $h: \mathbb{R}^{+}\rightarrow\mathbb{R}$ by
\begin{equation}\label{equ4-7}
h(t):=\frac{a}{2}t^{2}+\frac{b}{4}t^{4}-\frac{C^{p}(s, p)}{p}c^{p(1-\vartheta_{s, p})}t^{p\vartheta_{s, p}}-\frac{\mu C^{q}(s, q)}{q}c^{q(1-\vartheta_{s, q})}t^{q\vartheta_{s, q}}.
\end{equation}
Since $\mu>0$, $q\vartheta_{s, q}<2$ and $p\vartheta_{s, p}>4$, we can see that $h(0^{+})=0^{-}$ and $h(+\infty)=-\infty$.

\begin{lemma}\label{lem4-3}
Let $2<q<2+\frac{4s}{3}$, $2+\frac{8s}{3}<p<2_{s}^{\ast}$ and $0<\mu<\mu_{2}$. Then the function $h$ has a local strict minimum at negative level, a global strict maximum at positive level, and no other critical points, and there exist $0<R_{0}<R_{1}$, both depending on $c$ and $\mu$, such that $h(R_{0})=0=h(R_{1})$ and $h(t)>0$ iff $t\in(R_{0}, R_{1})$. Here $\mu_{2}$ is defined in \eqref{equ1-8}.
\end{lemma}
\begin{proof}
Letting $\hat{a}=\frac{a}{2}$, $\hat{b}=\frac{b}{4}$, $\hat{c}=\frac{C^{p}(s, p)}{p}c^{p(1-\vartheta_{s, p})}$, $\hat{d}=\frac{\mu C^{q}(s, q)}{q}c^{q(1-\vartheta_{s, q})}$, $\hat{q}=q\vartheta_{s, q}$ and $\hat{p}=p\vartheta_{s, p}$ in Lemma \ref{lem2-8}, then we only prove that
\begin{equation*}
\left[(\frac{8(4-\hat{q})}{\hat{p}(\hat{p}-2)(\hat{p}-\hat{q})})^{\frac{4-\hat{q}}{\hat{p}-4}}-(\frac{8(4-\hat{q})}{\hat{p}(\hat{p}-2)(\hat{p}-\hat{q})})^{\frac{\hat{p}-\hat{q}}{\hat{p}-4}}\right]\left[\frac{\hat{a}}{\hat{d}}(\frac{\hat{b}}{\hat{c}})^{\frac{2-\hat{q}}{\hat{p}-4}}+\frac{1}{\hat{d}}\frac{\hat{b}^{\frac{\hat{p}-\hat{q}}{\hat{p}-4}}}{\hat{c}^{\frac{4-\hat{q}}{\hat{p}-4}}}\right]>1
\end{equation*}
holds. In fact, by simple calculations, we can see that the above inequality is equivalent to
\begin{equation*}
\mu<\frac{qA(p, q, s)}{C^{q}(s, q)}\left[\frac{\frac{a}{2}(\frac{bp}{4C^{p}(s, p)})^{\frac{2-q\vartheta_{s, q}}{p\vartheta_{s, p}-4}}}{c^{q(1-\vartheta_{s, q})+\frac{p(1-\vartheta_{s, p})(2-q\vartheta_{s, q})}{p\vartheta_{s, p}-4}}}+\frac{(\frac{b}{4})^{\frac{p\vartheta_{s, p}-q\vartheta_{s, q}}{p\vartheta_{s, p}-4}}(\frac{p}{C^{p}(s, p)})^{\frac{4-q\vartheta_{s, q}}{p\vartheta_{s, p}-4}}}{c^{q(1-\vartheta_{s, q})+\frac{p(1-\vartheta_{s, p})(4-q\vartheta_{s, q})}{p\vartheta_{s, p}-4}}}\right]=\mu_{2},
\end{equation*}
hence the conclusion follows provided $0<\mu<\mu_{2}$.
\end{proof}

\begin{lemma}\label{lem4-4}
Let $2<q<2+\frac{4s}{3}$, $2+\frac{8s}{3}<p<2_{s}^{\ast}$ and $0<\mu<\min\{\mu_{1}, \mu_{2}\}$. For any $u\in M_{c}$, the function $\mathcal{J}_{u}^{\mu}$ has exactly two critical points $s_{u}<t_{u}\in\mathbb{R}$ and two zeros $c_{u}<d_{u}\in\mathbb{R}$ with $s_{u}<c_{u}<t_{u}<d_{u}$. Moreover:\\
$(i)$ $s_{u}\star u\in\mathcal{P}_{c, \mu}^{+}$ and $t_{u}\star u\in\mathcal{P}_{c, \mu}^{-}$, and if $\tau\star u\in\mathcal{P}_{c, \mu}$, then either $\tau=s_{u}$ or $\tau=t_{u}$.\\
$(ii)$ $\|(-\Delta)^{\frac{s}{2}}(\tau\star u)\|_{2}\leq R_{0}$ for any $\tau\leq c_{u}$ and
\begin{equation}\label{equ4-8}
E_{\mu}(s_{u}\star u)=\min\{E_{\mu}(\tau\star u): \tau\in\mathbb{R} \quad and \quad \|(-\Delta)^{\frac{s}{2}}(\tau\star u)\|_{2}<R_{0}\}<0.
\end{equation}
$(iii)$ We have
\begin{equation}\label{equ4-9}
E_{\mu}(t_{u}\star u)=\max\{E_{\mu}(\tau\star u): \tau\in\mathbb{R}\}>0,
\end{equation}
and $\mathcal{J}_{u}^{\mu}$ is strictly decreasing on $(t_{u}, +\infty)$.\\
$(iv)$ The maps $u\in M_{c, r}\mapsto s_{u}\in\mathbb{R}$ and $u\in M_{c, r}\mapsto t_{u}\in\mathbb{R}$ are of class $C^{1}$.
\end{lemma}
\begin{proof}
For any $u\in M_{c, r}$, by Lemma \ref{lem2-5}, we know that $\tau\star u\in\mathcal{P}_{c, \mu}$ if and only if $\tau$ is a critical point of $\mathcal{J}_{u}^{\mu}$. Now, we prove that $\mathcal{J}_{u}^{\mu}$ has at least two critical points. Indeed, by \eqref{equ4-6}-\eqref{equ4-7}, we can see that
\begin{equation*}
\begin{split}
\mathcal{J}_{u}^{\mu}(\tau)=E_{\mu}(\tau\star u)\geq h(\|(-\Delta)^{\frac{s}{2}}(\tau\star u)\|_{2})=h(e^{s\tau}\|(-\Delta)^{\frac{s}{2}}u\|_{2})
\end{split}
\end{equation*}
By Lemma \ref{lem4-3}, we have that $h>0$ on an open interval $(R_{0}, R_{1})$. Hence, $\mathcal{J}_{u}^{\mu}$ is positive on $(C(R_{0}), C(R_{1}))$, with $(C(R_{0}), C(R_{1})):=(\frac{1}{s}\ln(R_{0}/\|(-\Delta)^{\frac{s}{2}}u\|_{2}), \frac{1}{s}\ln(R_{1}/\|(-\Delta)^{\frac{s}{2}}u\|_{2}))$. Moreover, it is easy to see that $\mathcal{J}_{u}^{\mu}(-\infty)=0^{-}$, $\mathcal{J}_{u}^{\mu}(+\infty)=-\infty$. Then, we can see that $\mathcal{J}_{u}^{\mu}$ has at least two critical
points $s_{u}<t_{u}$, more precisely, $s_{u}\in(0, C(R_{0}))$ is a local minimum point at the negative level and $t_{u}>s_{u}$ is a global maximum point at the positive level. Further, similar to the argument of Lemma 6.6 in \cite{LLLL}, we can see that $\mathcal{J}_{u}^{\mu}$ has no other critical points. Next, let's verify that $(i)-(iv)$ holds.

$(iii)$ According to the above analysis, we know that $\mathcal{J}_{u}^{\mu}$ has exactly two critical points: $s_{u}$ is local minimum on $(-\infty, C(R_{0}))$ at the negative level and $t_{u}$ is a global maximum at the positive level, which give \eqref{equ4-9}. Moreover, by monotonicity of $\mathcal{J}_{u}^{\mu}$, $\mathcal{J}_{u}^{\mu}(-\infty)=0^{-}$ and $\mathcal{J}_{u}^{\mu}(+\infty)=-\infty$, we can get that $\mathcal{J}_{u}^{\mu}$ has exactly two zeros $c_{u}<d_{u}$, with $s_{u}<c_{u}<t_{u}<d_{u}$ and $\mathcal{J}_{u}^{\mu}$ is strictly decreasing on $(t_{u}, +\infty)$.

$(ii)$ Since $s_{u}<C(R_{0})$, by directly computation, we deduce that
\begin{equation*}
\begin{split}
&\|(-\Delta)^{\frac{s}{2}}(s_{u}\star u)\|_{2}=e^{ss_{u}}\|(-\Delta)^{\frac{s}{2}}u\|_{2}\\
&<e^{sC(R_{0})}\|(-\Delta)^{\frac{s}{2}}u\|_{2}=R_{0},
\end{split}
\end{equation*}
which implies that \eqref{equ4-8} holds.

$(i)$ By Lemma \ref{lem2-5}, $(\mathcal{J}_{u}^{\mu})'(s_{u})=(\mathcal{J}_{u}^{\mu})'(t_{u})=0$, we can get $s_{u}\star u\in\mathcal{P}_{c, \mu}$ and  $t_{u}\star u\in\mathcal{P}_{c, \mu}$, since $\mathcal{J}_{u}^{\mu}$ has exactly two critical points, so, if $\tau\star u\in\mathcal{P}_{c, \mu}$, then $\tau\in\{s_{u}, t_{u}\}$. Moreover, since $\mathcal{P}_{c, \mu}^{0}=\emptyset$ by Lemma \ref{lem4-1}, then $(\mathcal{J}_{u}^{\mu})''(s_{u})>0$, thai is $s_{u}\star u\in\mathcal{P}_{c, \mu}^{+}$. In the same way, we can get $t_{u}\star u\in\mathcal{P}_{c, \mu}^{-}$.

$(iv)$ We apply the implicit function theorem on the $C^{1}$ function $\Phi(\tau, u)=(\mathcal{J}_{u}^{\mu})'(\tau)$. We use that $\Phi(\tau, u)=0$, that $\partial_{s}\Phi(s_{u}, u)=(\mathcal{J}_{u}^{\mu})''(s_{u})>0$ and the fact that it is not possible to pass with continuity from $\mathcal{P}_{c, \mu}^{+}$ to $\mathcal{P}_{c, \mu}^{-}$ since $\mathcal{P}_{c, \mu}^{0}=\emptyset$. Hence, $u\mapsto s_{u}$ is of class $C^{1}$. The same argument proves that $u\mapsto t_{u}$ is of class $C^{1}$.
\end{proof}

For $k>0$, set
\begin{equation*}
A_{k}:=\{u\in M_{c}: \|(-\Delta)^{\frac{s}{2}}u\|_{2}<k\} \quad and \quad m(c, \mu):=\inf_{u\in A_{R_{0}}}E_{\mu}(u).
\end{equation*}

From Lemma \ref{lem4-4}, we have
\begin{lemma}\label{lem4-5}
Let $2<q<2+\frac{4s}{3}$, $2+\frac{8s}{3}<p<2_{s}^{\ast}$ and $0<\mu<\min\{\mu_{1}, \mu_{2}\}$. Then $\mathcal{P}_{c, \mu}^{+}\subseteq\{u\in M_{c}: \|(-\Delta)^{\frac{s}{2}}u\|_{2}<R_{0}\}$ and $\sup_{\mathcal{P}_{c, \mu}^{+}}E_{\mu}\leq0\leq\inf_{\mathcal{P}_{c, \mu}^{-}}E_{\mu}$.
\end{lemma}

Moreover, we have
\begin{lemma}\label{lem4-6}
Let $2<q<2+\frac{4s}{3}$, $2+\frac{8s}{3}<p<2_{s}^{\ast}$ and $0<\mu<\min\{\mu_{1}, \mu_{2}\}$. It results that $m(c, \mu)\in(-\infty, 0)$ and
\begin{equation*}
m(c, \mu)=\inf_{\mathcal{P}_{c, \mu}}E_{\mu}=\inf_{\mathcal{P}_{c, \mu}^{+}}E_{\mu} \quad and \quad m(c, \mu)<\inf_{\overline{A_{R_{0}}}\backslash A_{R_{0}-\rho}}E_{\mu}(u).
\end{equation*}
for $\rho>0$ small enough.
\end{lemma}

\begin{proof}
For $u\in A_{R_{0}}$, we have
\begin{equation*}
E_{\mu}(u)\geq h(\|(-\Delta)^{\frac{s}{2}}u\|_{2})\geq\min_{t\in[0, R_{0}]}h(t)>-\infty,
\end{equation*}
which implies that $m(c, \mu)>-\infty$. Note that $u\in M_{c}$, then $\|(-\Delta)^{\frac{s}{2}}(\tau\star u)\|_{2}=e^{s\tau}\|(-\Delta)^{\frac{s}{2}}u\|_{2}<R_{0}$ for $\tau\ll-1$. Hence
\begin{equation*}
\begin{split}
E_{\mu}(\tau\star u)
&=\frac{ae^{2s\tau}}{2}\|(-\Delta)^{\frac{s}{2}}u\|_{2}^{2}+\frac{be^{4s\tau}}{4}\|(-\Delta)^{\frac{s}{2}}u\|_{2}^{4}\\
&-\mu\frac{e^{3\tau(\frac{q}{2}-1)}}{q}\|u\|_{q}^{q}-\frac{e^{3\tau(\frac{p}{2}-1)}}{p}\|u\|_{p}^{p}<0, \quad \hbox{ $\tau\ll-1$,}
\end{split}
\end{equation*}
which implies that $m(c, \mu)<0$.

Since $\mathcal{P}_{c, \mu}^{+}\subseteq A_{R_{0}}$ by Lemma \ref{lem4-5}, then $m(c, \mu)\leq\inf_{\mathcal{P}_{a, \mu}^{+}}E_{\mu}$. In addition, from Lemma \ref{lem4-4}, if $u\in A_{R_{0}}$, then $s_{u}\star u\in\mathcal{P}_{c, \mu}^{+}\subseteq A_{R_{0}}$ and
\begin{equation*}
E_{\mu}(s_{u}\star u)=\min\{E_{\mu}(\tau\star u): \tau\in\mathbb{R} \quad and \quad \|(-\Delta)^{\frac{s}{2}}(\tau\star u)\|_{2}<R_{0}\}\leq E_{\mu}(u),
\end{equation*}
which implies that $\inf_{\mathcal{P}_{c, \mu}^{+}}E_{\mu}\leq m(c, \mu)$. On the other hand, since $E_{\mu}(u)>0$ on $\mathcal{P}_{c, \mu}^{-}$, so it is easy to check that $\inf_{\mathcal{P}_{c, \mu}^{+}}E_{\mu}=\inf_{\mathcal{P}_{c, \mu}}E_{\mu}$.

Now, according to the continuity of $h$, there exists $\rho>0$ such that $h(t)\geq\frac{m(c, \mu)}{2}$ if $t\in[R_{0}-\rho, R_{0}]$. Therefore, we can see that
\begin{equation*}
E_{\mu}(u)\geq h(\|(-\Delta)^{\frac{s}{2}}u\|_{2})\geq\frac{m(c, \mu)}{2}>m(c, \mu),
\end{equation*}
for $u\in M_{c}$ with $R_{0}-\rho\leq\|(-\Delta)^{\frac{s}{2}}u\|_{2}\leq R_{0}$, that is $m(c, \mu)<\inf_{\overline{A_{R_{0}}}\backslash A_{R_{0}-\rho}}E_{\mu}(u)$.
\end{proof}

\begin{lemma}\label{lem4-7}
Let $2<q<2+\frac{4s}{3}$, $2+\frac{8s}{3}<p<2_{s}^{\ast}$ and $0<\mu<\min\{\mu_{1}, \mu_{2}\}$. Suppose that $E_{\mu}(u)<m(c, \mu)$. Then the value $t_{u}$ defined by Lemma \ref{lem4-4} is negative.
\end{lemma}
\begin{proof}
The proof is similar to that of Lemma 6.10 in \cite{LLLL}.

\end{proof}

\begin{lemma}\label{lem4-8}
Let $2<q<2+\frac{4s}{3}$, $2+\frac{8s}{3}<p<2_{s}^{\ast}$ and $0<\mu<\min\{\mu_{1}, \mu_{2}\}$. It holds that
\begin{equation*}
\tilde{\sigma}(c, \mu):=\inf_{\mathcal{P}_{c, \mu}^{-}}E_{\mu}>0.
\end{equation*}
\end{lemma}
\begin{proof}
The proof is similar to that of Lemma 6.11 in \cite{LLLL}.

\end{proof}

Now, we define the minimax class
\begin{equation*}
\Gamma:=\{\gamma\in C([0, 1], M_{c, r}): \gamma(0)\in\mathcal{P}_{c, \mu}^{+}, E_{\mu}(\gamma(1))\leq2m(c, \mu)\}.
\end{equation*}
Clearly, $\Gamma\neq\emptyset$. In fact, for any $u\in M_{c, r}$, we have $s_{u}\star u\in\mathcal{P}_{c, \mu}^{+}$ by Lemma \ref{lem4-4}, $E_{\mu}(\tau\star u)\rightarrow-\infty$ as $\tau\rightarrow+\infty$ and $\tau\mapsto\tau\star u$ is continuous. Then, we only need to let $\gamma(t)=\tau_{t}\star u$.

Next, we can define the minimax value
\begin{equation*}
\sigma(c, \mu)=\inf_{\gamma\in\Gamma}\max_{t\in[0, 1]}E_{\mu}(\gamma(t)).
\end{equation*}

\text{\bf Proof of Theorem \ref{the1-1}-(1).}Let $\{v_{n}\}$ be a minimizing sequence for $m(c, \mu):=\inf_{u\in A_{R_{0}}}E_{\mu}(u)$. By Lemma \ref{lem2-3}, we get another function in $A_{R_{0}}$ with $E_{\mu}(|v_{n}|^{\ast})\leq E_{\mu}(v_{n})$. Thus, we can assume that $v_{n}\in M_{c}$ is  nonnegative and radially decreasing for every $n$. By Lemma \ref{lem4-4}, there exists a sequence $s_{v_{n}}$ such that $s_{v_{n}}\star v_{n}\in\mathcal{P}_{a, \mu}^{+}$ and
\begin{equation*}
E_{\mu}(s_{v_{n}}\star v_{n})=\min\{E_{\mu}(\tau\star v_{n}): \tau\in\mathbb{R} \quad and \quad \|(-\Delta)^{\frac{s}{2}}(\tau\star v_{n})\|_{2}<R_{0}\}<E_{\mu}(v_{n}),
\end{equation*}
and
\begin{equation*}
\|(-\Delta)^{\frac{s}{2}}(s_{v_{n}}\star v_{n})\|_{2}< R_{0}.
\end{equation*}

Let $u_{n}=s_{v_{n}}\star v_{n}\subset A_{R_{0}}$. By Lemma \ref{lem4-6}, we have $\|(-\Delta)^{\frac{s}{2}}(s_{v_{n}}\star v_{n})\|_{2}< R_{0}-\rho$. Using the Ekeland's variational principle, we may assume that $\{u_{n}\}$ is a Palais-Smale sequence for $E_{\mu}|_{M_{c, r}}$ and $P_{\mu}(u_{n})=0$. So, $\{u_{n}\}$ satisfies all the assumptions of Proposition \ref{lem3-1}. Therefore, up to a subsequence, $u_{n}\rightarrow\tilde{u}_{\mu}$ in $H^{s}(\mathbb{R}^{3})$, where $\tilde{u}_{\mu}$ is an interior local minimizer for $E_{\mu}|_{A_{R_{0}}}$ and $\tilde{u}_{\mu}$ is a radial solution to \eqref{equ1-1} for some $\tilde{\lambda}_{\mu}<0$. It is easy to know that $\tilde{u}_{\mu}$ is nonnegative and radially deceasing. Assume that there exists $x_{0}\in\mathbb{R}^{3}$ such that $\tilde{u}_{\mu}(x_{0})=0$, then from $\tilde{u}_{\mu}\geq0$ and $\tilde{u}_{\mu}\not\equiv0$, we get
\begin{equation*}
(-\Delta)^{s}\tilde{u}_{\mu}(x_{0})=-\frac{1}{2}C(3, s)\int_{\mathbb{R}^{3}}\frac{\tilde{u}_{\mu}((x_{0})+y)+\tilde{u}_{\mu}((x_{0})-y)}{|x_{0}-y|^{3+2s}}dxdy<0.
\end{equation*}
However, it is easy to see that
\begin{equation*}
(-\Delta)^{s}\tilde{u}_{\mu}(x_{0})=\frac{(\lambda \tilde{u}_{\mu}(x_{0})+\mu(\tilde{u}_{\mu}(x_{0}))^{q-2}\tilde{u}_{\mu}(x_{0})+(\tilde{u}_{\mu}(x_{0}))^{p-2}\tilde{u}_{\mu}(x_{0}))}{a+b\int_{\mathbb{R}^{3}}|(-\Delta)^{\frac{s}{2}}\tilde{u}_{\mu}(x_{0})|^{2}dx}=0,
\end{equation*}
which gives a contradiction. Thus, $\tilde{u}_{\mu}>0$. By Lemma \ref{lem4-6}, we know that $\tilde{u}_{\mu}$ is a ground state.

Finally, we need to show that any other ground state is a local minimizer for $E_{\mu}$ on $A_{R_{0}}$. In fact, let $u$ be a critical point of $E_{\mu}|_{M_{c}}$ with $E_{\mu}(u)=m(c, \mu)=\inf_{\mathcal{P}_{c, \mu}}E_{\mu}$. Since $E_{\mu}(u)<0<\inf_{\mathcal{P}_{c, \mu}^{-}}E_{\mu}$, necessarily $u\in\mathcal{P}_{c, \mu}^{+}$. Then we can get that $\mathcal{P}_{c, \mu}^{+}\subset A_{R_{0}}$ by Lemma \ref{lem4-5}. Therefore, $\|(-\Delta)^{\frac{s}{2}}u\|_{2}\leq R_{0}$, and as a consequence $u$ is a local minimizer for $E_{\mu}|_{A_{R_{0}}}$.

\text{\bf Proof of Theorem \ref{the1-1}-(2).}We divide the proof into five steps.

{\bf Step 1.}\ In order to use Lemma \ref{lem2-7}, let us set
\begin{equation*}
\mathcal{F}:=\Gamma, \quad A:=\gamma([0, 1]), \quad F:=\mathcal{P}^{-}_{c, \mu}, \quad B:=\mathcal{P}^{+}_{c, \mu}\cup E_{\mu}^{2m(c, \mu)},
\end{equation*}
where $E_{\mu}^{2m(c, \mu)}:=\{u\in M_{c, r}: E_{\mu}(u)\leq2m(c, \mu)\}$. Now, according to Definition \ref{lem2-6}, we show that $\mathcal{F}$ is homotopy-stable family with extended boundary $B$: for any $\gamma\in\Gamma$ and $\eta\in C([0, 1]\times M_{c, r}, M_{c, r})$ satisfying $\eta(t, u)=u$, $(t, u)\in(0\times M_{c, r})\cup([0, 1]\times B)$, we want to get $\eta(1, \gamma(t))\in\Gamma$. In fact, let $\tilde{\gamma}(t)=\eta(1, \gamma(t))$, then $\tilde{\gamma}(0)=\eta(1, \gamma(0))=\gamma(0)\in \mathcal{P}_{c, \mu}^{+}$, $\tilde{\gamma}(1)=\eta(1, \gamma(1))=\gamma(1)\in E_{\mu}^{2m(c, \mu)}$. So, we get $\eta(1, \gamma(t))\in\Gamma$.

{\bf Step 2.}\ We prove that the condition \eqref{equ2-2} in Lemma \ref{lem2-7} holds.\\
By Lemma \ref{lem4-5} and Lemma \ref{lem4-8}, we can see that $F\cap B=\emptyset$, that is $F\backslash B=F$. Next, we show that
\begin{equation*}
A\cap(F\backslash B)=A\cap F=\gamma([0, 1])\cap\mathcal{P}^{-}_{c, \mu}\neq\emptyset.
\end{equation*}
Indeed, since $\gamma\in\Gamma$, then $\gamma(0)\in\mathcal{P}_{c, \mu}^{+}$, by Lemma \ref{lem4-4}, we know that $s_{\gamma(0)}=0$. Moreover, $E_{\mu}(\gamma(1))\leq2m(c, \mu)<m(c, \mu)$ since Lemma \ref{lem4-6}, then Lemma \ref{lem4-7} implies that $t_{\gamma(1)}<0$.  Furthermore, by Lemma \ref{lem4-4}, we know that $t_{\gamma(\tau)}$ is continuous in $\tau$. It follows that for every $\gamma\in\Gamma$ there exists $\gamma(\tau)\in(0, 1)$ such that $t_{\gamma(\tau)}=0$, that is, $\gamma(\tau)\in\mathcal{P}_{c, \mu}^{-}$. Therefore, $A\cap(F\backslash B)\neq\emptyset$.

{\bf Step 3.}\ We prove that the condition \eqref{equ2-3} in Lemma \ref{lem2-7} holds.\\
In fact, we need to prove that
\begin{equation*}
\inf_{\mathcal{P}^{-}_{c, \mu}}E_{\mu}\geq\sigma(c, \mu)\geq\sup_{\mathcal{P}^{+}_{c, \mu}\cup E_{\mu}^{2m(c, \mu)}}E_{\mu}.
\end{equation*}
By step 2, we know that $\gamma([0, 1])\cap\mathcal{P}^{-}_{c, \mu}\neq\emptyset$, then we have
\begin{equation*}
\max_{t\in[0, 1]}E_{\mu}(\gamma(t))\geq\inf_{\mathcal{P}^{-}_{c, \mu}}E_{\mu},
\end{equation*}
which implies that $\sigma(c, \mu)\geq \tilde{\sigma}(c, \mu)$. Moreover, if $u\in\mathcal{P}^{-}_{c, \mu}$, then for $s_{1}\gg1$,
\begin{equation*}
\gamma_{u}: \tau\in[0, 1]\mapsto((1-\tau)s_{u}+\tau s_{1})\star u\in M_{c, r}
\end{equation*}
is a path in $\Gamma$. Since $u\in\mathcal{P}^{-}_{c, \mu}$, by Lemma \ref{lem4-4}, we have $t_{u}=0$ is a global maximum point for $\mathcal{J}_{\mu}^{u}$, hence
\begin{equation*}
E_{\mu}(u)\geq\max_{t\in[0, 1]}E_{\mu}(\gamma_{u}(t))\geq\sigma(c, \mu),
\end{equation*}
which implies that $\tilde{\sigma}(c, \mu)\geq\sigma(c, \mu)$. So, we have $\tilde{\sigma}(c, \mu)=\sigma(c, \mu)>0$. In addition, by Lemma \ref{lem4-5}, we can see that $E_{\mu}(u)\leq 0$ for any $u\in\mathcal{P}^{+}_{c, \mu}\cup E_{\mu}^{2m(c, \mu)}$. Hence, we get \eqref{equ2-3} holds.

{\bf Step 4.}\ By Step 1, Step 2 and Step 3, we can see that all the conditions in Lemma \ref{lem2-7} holds. Therefore, by Lemma \ref{lem2-7}, we obtain a Palais-Smale sequence $\{u_{n}\}$ for $E_{\mu}|_{M_{c, r}}$ at the level $\sigma(c, \mu)>0$ and dist$(u_{n}, \mathcal{P}^{-}_{c, \mu})\rightarrow0$, i.e. $P_{\mu}(u_{n})\rightarrow0$.

{\bf Step 5.}\ By Step 4 and Proposition \ref{lem3-1}, up to a subsequence, $u_{n}\rightarrow\hat{u}_{\mu}$ in $H^{s}(\mathbb{R}^{3})$, where $\hat{u}_{\mu}\in M_{c, r}$ is a nonnegative radial solution to \eqref{equ1-1} for some $\hat{\lambda}_{\mu}<0$. Similar to the arguments of Theorem \ref{the1-1}-$(1)$, we know that $\hat{u}_{\mu}>0$.

Next,we are concerned with the behavior of the ground states in Theorem \ref{the1-1} as $\mu\rightarrow0^{+}$. The following two lemmas can be obtain as the proof of Lemma \ref{lem4-1} and Lemma 6.14 in \cite{LLLL}.
\begin{lemma}\label{lem4-9}
Let $2+\frac{8s}{3}<p<2_{s}^{\ast}$ and $\mu=0$. Then $\mathcal{P}_{c, \mu}^{0}=\emptyset$ and $\mathcal{P}_{c, \mu}$ is a smooth manifold of codimension $2$ in $M_{c}$.
\end{lemma}

\begin{lemma}\label{lem4-10}
Let $2+\frac{8s}{3}<p<2_{s}^{\ast}$ and $\mu=0$. For every $u\in M_{c}$, there exists a unique $t_{u}\in\mathbb{R}$ such that $t_{u}\star u\in\mathcal{P}_{c, \mu}$. $t_{u}$ is the unique critical point of the function $\mathcal{J}_{u}^{\mu}$ and is a strict maximum point at positive level. Moreover,\\
$(1)$ $\mathcal{P}_{c, \mu}=\mathcal{P}_{c, \mu}^{-}$.\\
$(2)$ $\mathcal{J}_{u}^{\mu}$ is strictly decreasing and concave on $(t_{u}, \infty)$.\\
$(3)$ The map $u\in M_{c}\mapsto t_{u}\in\mathbb{R}$ is of class $C^{1}$.\\
$(4)$ If $P_{\mu}(u)<0$, then $t_{u}<0$.
\end{lemma}
\begin{lemma}\label{lem4-11}
Let $2+\frac{8s}{3}<p<2_{s}^{\ast}$ and $\mu=0$. Then $m(c, 0):=\inf_{u\in\mathcal{P}_{c, 0}}E_{0}(u)>0$.
\end{lemma}
\begin{proof}
By Lemma \ref{lem2-1} and $P_{0}=0$, we can see that
\begin{equation*}
a\|(-\Delta)^{\frac{s}{2}}u\|_{2}^{2}+b\|(-\Delta)^{\frac{s}{2}}u\|_{2}^{4}=\vartheta_{s, p}\|u\|_{p}^{p}\leq\vartheta_{s, p}C^{p}(s, p)\|(-\Delta)^{\frac{s}{2}}u\|_{2}^{p\vartheta_{s, p}}c^{p(1-\vartheta_{s, p})},
\end{equation*}
which implies that $\inf_{\mathcal{P}_{c, 0}}\|(-\Delta)^{\frac{s}{2}}u\|_{2}\geq C>0$ since $p\vartheta_{s, p}>4$. Then, using $P_{0}=0$, we have
\begin{equation*}
\inf_{\mathcal{P}_{c, 0}}E_{0}(u)=\inf_{\mathcal{P}_{c, 0}}\left\{(\frac{a}{2}-\frac{a}{p\vartheta_{s, p}})\|(-\Delta)^{\frac{s}{2}}u\|_{2}^{2}+(\frac{b}{4}-\frac{b}{p\vartheta_{s, p}})\|(-\Delta)^{\frac{s}{2}}u\|_{2}^{4}\right\}\geq C>0.
\end{equation*}

\end{proof}

\begin{lemma}\label{lem4-12}
Let $2+\frac{8s}{3}<p<2_{s}^{\ast}$ and $\mu=0$. There exists $k>0$ sufficiently small such
\begin{equation*}
0<\sup_{\overline{A_{k}}}E_{0}<m(c, 0) \quad and \quad u\in\overline{A_{k}}\Rightarrow E_{0}(u), P_{0}(u)>0,
\end{equation*}
where $A_{k}=\{u\in M_{c}: \|(-\Delta)^{\frac{s}{2}}u\|_{2}<k\}$.
\end{lemma}
\begin{proof}
For $u\in\overline{A_{k}}$, by Lemma \ref{lem2-1}, we have
\begin{equation*}
E_{0}(u)\geq\frac{b}{4}\|(-\Delta)^{\frac{s}{2}}u\|_{2}^{4}-\frac{C^{p}(s, p)}{p}c^{p(1-\vartheta_{s, p})}\|(-\Delta)^{\frac{s}{2}}u\|_{2}^{p\vartheta_{s, p}}>0,
\end{equation*}
\begin{equation*}
P_{0}(u)\geq b\|(-\Delta)^{\frac{s}{2}}u\|_{2}^{4}-\vartheta_{s, p}C^{p}(s, p)c^{p(1-\vartheta_{s, p})}\|(-\Delta)^{\frac{s}{2}}u\|_{2}^{p\vartheta_{s, p}}>0,
\end{equation*}
with $k$ small enough since $p\vartheta_{s, p}>4$. Now if replacing $k$ with a smaller quantity, recalling that $m(c, 0)>0$ by Lemma \ref{lem4-11}, we also have
\begin{equation*}
E_{0}(u)\leq\frac{a}{2}\|(-\Delta)^{\frac{s}{2}}u\|_{2}^{2}+\frac{b}{4}\|(-\Delta)^{\frac{s}{2}}u\|_{2}^{4}<m(c, 0).
\end{equation*}
\end{proof}

Now, by Lemmas \ref{lem4-9}-\ref{lem4-12} and using the same arguments in Section 7 in \cite{SSSSSS}, we have an immediate result:
\begin{lemma}\label{lem4-13}
Let $2+\frac{8s}{3}<p<2_{s}^{\ast}$ and $\mu=0$. There exists a positive radial critical point $u_{0}$ for $E_{0}|_{M_{c}}$ at a positive level
\begin{equation*}
m_{r}(c, 0)=m(c, 0)=\inf_{\mathcal{P}_{c, 0}}E_{0}=E_{0}(u_{0}),
\end{equation*}
and as a result $u_{0}$ is the unique ground state of $E_{0}|_{M_{c}}$. Here $m_{r}(c, 0):=\inf_{\mathcal{P}_{c, 0}\cap M_{c, r}}E_{0}$.
\end{lemma}

\begin{lemma}\label{lem4-14}
Let $2<q<2+\frac{4s}{3}$, $2+\frac{8s}{3}<p<2_{s}^{\ast}$ and $0<\mu<\min\{\mu_{1}, \mu_{2}\}$. Then
\begin{equation*}
\inf_{\mathcal{P}_{c, \mu}^{-}\cap M_{c, r}}E_{\mu}=\inf_{u\in M_{c, r}}\max_{\tau\in\mathbb{R}}E_{\mu}(\tau\star u),
\end{equation*}
and
\begin{equation*}
\inf_{\mathcal{P}_{c, 0}^{-}\cap M_{c, r}}E_{0}=\inf_{u\in M_{c, r}}\max_{\tau\in\mathbb{R}}E_{0}(\tau\star u).
\end{equation*}
\end{lemma}

\begin{proof}
For $u\in\mathcal{P}_{c, 0}^{-}\cap M_{c, r}$, by Lemma \ref{lem4-4}, we can get that $t_{u}=0$ and
\begin{equation*}
E_{\mu}(u)=\max_{\tau\in\mathbb{R}}E_{\mu}(\tau\star u)\geq \inf_{v\in M_{c, r}}\max_{\tau\in\mathbb{R}}E_{\mu}(\tau\star v).
\end{equation*}
Moreover, if $u\in M_{c, r}$, then $t_{u}\star u\in\mathcal{P}_{c, \mu}^{-}\cap M_{c, r}$, so
\begin{equation*}
\max_{\tau\in\mathbb{R}}E_{\mu}(\tau\star u)=E_{\mu}(t_{u}\star u)\geq\inf_{\mathcal{P}_{c, \mu}^{-}\cap M_{c, r}}E_{\mu}.
\end{equation*}
Similarly, we can prove $\inf_{\mathcal{P}_{c, 0}^{-}\cap M_{c, r}}E_{0}=\inf_{u\in M_{c, r}}\max_{\tau\in\mathbb{R}}E_{0}(\tau\star u)$ by Lemma \ref{lem4-10}.

\end{proof}

\begin{lemma}\label{lem4-15}
Let $2<q<2+\frac{4s}{3}$, $2+\frac{8s}{3}<p<2_{s}^{\ast}$ and $0\leq\mu_{\alpha}<\mu_{\beta}<\min\{\mu_{1}, \mu_{2}\}$. Then it holds that $\sigma(c, \mu_{\beta})\leq\sigma(c, \mu_{\alpha})\leq m(c, 0)$.

\end{lemma}
\begin{proof}
Using the step $3$ in the proof of Theorem \ref{the1-1}-$(2)$, we know that $\sigma(c, \mu)=\inf_{\mathcal{P}_{c, \mu}^{-}\cap M_{c, r}}E_{\mu}$.
By Lemmas \ref{lem4-13}-\ref{lem4-14}, we can see that
\begin{equation*}
\sigma(c, \mu_{\alpha})=\inf_{u\in M_{c, r}}\max_{\tau\in\mathbb{R}}E_{\mu_{\alpha}}(\tau\star u)\leq\inf_{u\in M_{c, r}}\max_{\tau\in\mathbb{R}}E_{0}(\tau\star u)=m_{r}(c, 0)=m(c, 0),
\end{equation*}
and
\begin{equation*}
\sigma(c, \mu_{\beta})\leq\max_{\tau\in\mathbb{R}}E_{\mu_{\beta}}(\tau\star \hat{u}_{\mu_{\alpha}})\leq\max_{\tau\in\mathbb{R}}E_{\mu_{\alpha}}(\tau\star \hat{u}_{\mu_{\alpha}})=E_{\mu_{\alpha}}(\hat{u}_{\mu_{\alpha}})=\sigma(c, \mu_{\alpha}).
\end{equation*}
\end{proof}

\text{\bf Proof of Theorem \ref{the1-2}-(1).}From Lemma \ref{lem4-3}, it is easy to check that $R_{0}=R_{0}(c, \mu)\rightarrow0$ as $\mu\rightarrow0^{+}$. Then $\|(-\Delta)^{\frac{s}{2}}\tilde{u}_{\mu}\|_{2}\rightarrow0$ and
\begin{equation*}
\begin{split}
&0>m(c, \mu)\geq\frac{a}{2}\|(-\Delta)^{\frac{s}{2}}\tilde{u}_{\mu}\|_{2}^{2}+\frac{b}{4}\|(-\Delta)^{\frac{s}{2}}\tilde{u}_{\mu}\|_{2}^{4}\\
&-\frac{\mu}{q}C^{q}(s, q)\|(-\Delta)^{\frac{s}{2}}\tilde{u}_{\mu}\|_{2}^{q\vartheta_{s, q}}c^{q(1-\vartheta_{s, q})}
-\frac{1}{p}C^{p}(s, p)\|(-\Delta)^{\frac{s}{2}}\tilde{u}_{\mu}\|_{2}^{p\vartheta_{s, p}}c^{p(1-\vartheta_{s, p})}\rightarrow0,
\end{split}
\end{equation*}
which implies that $m(c, \mu)\rightarrow0$.

\text{\bf Proof of Theorem \ref{the1-2}-(2).}Firstly, we prove that the family $\{\hat{u}_{\mu}: 0<\mu<\tilde{\mu}\}$ is bounded in $H^{s}(\mathbb{R}^{3})$. In fact, from Lemma \ref{lem2-1}, Lemma \ref{lem4-15} and $P_{\mu}(\hat{u}_{\mu})=0$, we can see that
\begin{equation*}
\begin{split}
m(c, 0)&\geq \sigma(c, \mu)=E_{\mu}(\hat{u}_{\mu})\\
&\geq(\frac{a}{2}-\frac{a}{p\vartheta_{s, p}})\|(-\Delta)^{\frac{s}{2}}\hat{u}_{\mu}\|_{2}^{2}+(\frac{b}{4}-\frac{b}{p\vartheta_{s, p}})\|(-\Delta)^{\frac{s}{2}}\hat{u}_{\mu}\|_{2}^{4}\\
&-\frac{\mu}{q}(1-\frac{q\vartheta_{s, q}}{p\vartheta_{s, p}})C^{q}(s, q)c^{q(1-\vartheta_{s, q})}\|(-\Delta)^{\frac{s}{2}}\hat{u}_{\mu}\|_{2}^{q\vartheta_{s, q}},
\end{split}
\end{equation*}
which implies that $\{\hat{u}_{\mu}\}$ is bounded in $H^{s}(\mathbb{R}^{3})$. Hence, up to a subsequence, we may assume that $\hat{u}_{\mu}\rightharpoonup \hat{u}\geq0$ in $H^{s}(\mathbb{R}^{3})$, $\hat{u}_{\mu}\rightarrow \hat{u}$ in $L^{r}(\mathbb{R}^{3})$, for any $r\in(2, 2_{s}^{\ast})$ and $\hat{u}_{\mu}\rightarrow \hat{u}$ a.e. in $\mathbb{R}^{3}$. Note that $\hat{u}_{\mu}$ solves
\begin{equation}\label{equ4-10}
\left(a+b\int_{\mathbb{R}^{3}}|(-\Delta)^{\frac{s}{2}}\hat{u}_{\mu}|^{2}dx\right)(-\Delta)^{s}\hat{u}_{\mu}=\hat{\lambda}_{\mu} \hat{u}_{\mu}+\mu|\hat{u}_{\mu}|^{q-2}\hat{u}_{\mu}+|\hat{u}_{\mu}|^{p-2}\hat{u}_{\mu} \quad \hbox{in $\mathbb{R}^3$,}
\end{equation}
with $\hat{\lambda}_{\mu}<0$. Then by $P_{\mu}(\hat{u}_{\mu})=0$, we have
\begin{equation*}
\hat{\lambda}_{\mu}c^{2}=\mu(\vartheta_{s, q}-1)\|\hat{u}_{\mu}\|_{q}^{q}+(\vartheta_{s, p}-1)\|\hat{u}_{\mu}\|_{p}^{p}.
\end{equation*}
So, by $\mu>0$ and $\vartheta_{s, p}, \vartheta_{s, q}\in(0, 1)$, up to a subsequence, we may assume that $\hat{\lambda}_{\mu}\rightarrow \hat{\lambda}\leq0$ satisfying $\hat{\lambda}c^{2}=(\vartheta_{s, p}-1)\|\hat{u}\|_{p}^{p}$ with $\hat{\lambda}=0$ if and only if $\hat{u}\equiv0$. Now, we show that $\hat{\lambda}<0$. Indeed, since $\hat{u}$ is a weak radial solution to
\begin{equation}\label{equ4-11}
(a+bB)(-\Delta)^{s}\hat{u}=\hat{\lambda} \hat{u}+|\hat{u}|^{p-2}\hat{u} \quad \hbox{in $\mathbb{R}^3$,}
\end{equation}
where $B:=\lim_{\mu\rightarrow0^{+}}\|(-\Delta)^{\frac{s}{2}}\hat{u}_{\mu}\|_{2}^{2}\geq\|(-\Delta)^{\frac{s}{2}}\hat{u}\|_{2}^{2}$. By Lemma \ref{lem4-15}, we can get
\begin{equation*}
\begin{split}
&-\frac{b}{4}\|(-\Delta)^{\frac{s}{2}}\hat{u}\|_{2}^{4}+(\frac{\vartheta_{s, p}}{2}-\frac{1}{p})\|\hat{u}\|_{p}^{p}\\
&\geq\lim_{\mu\rightarrow0^{+}}\left[-\frac{b}{4}\|(-\Delta)^{\frac{s}{2}}\hat{u}_{\mu}\|_{2}^{4}+(\frac{\vartheta_{s, p}}{2}-\frac{1}{p})\|\hat{u}_{\mu}\|_{p}^{p}-\mu(\frac{1}{q}-\frac{\vartheta_{s, q}}{2})\|\hat{u}_{\mu}\|_{q}^{q}\right]\\
&=\lim_{\mu\rightarrow0^{+}}E_{\mu}(\hat{u}_{\mu})=\lim_{\mu\rightarrow0^{+}}\sigma(c, \mu)\geq\sigma(c, \bar{\mu})>0,
\end{split}
\end{equation*}
which implies that $\hat{u}\not\equiv0$. Hence, $\hat{\lambda}<0$ and $B>0$. Similar to the arguments of Theorem \ref{the1-1}-$(1)$, we know that $\hat{u}>0$. Testing \eqref{equ4-10}-\eqref{equ4-11} with $\hat{u}_{\mu}-\hat{u}$, we have
\begin{equation*}
(a+bB)\int_{\mathbb{R}^{3}}|(-\Delta)^{\frac{s}{2}}(\hat{u}_{\mu}-\hat{u})|^{2}dx-\hat{\lambda}\int_{\mathbb{R}^{3}}|\hat{u}_{\mu}-\hat{u}|^{2}dx\rightarrow0,
\end{equation*}
which implies that $\hat{u}_{\mu}\rightarrow \hat{u}$ in $H^{s}(\mathbb{R}^{3})$ as $\mu\rightarrow0^{+}$. Then, we have
\begin{equation*}
E_{0}(\hat{u})=\frac{a}{2}\|(-\Delta)^{\frac{s}{2}}\hat{u}\|_{2}^{2}+\frac{b}{4}\|(-\Delta)^{\frac{s}{2}}\hat{u}\|_{2}^{4}-\frac{1}{p}\|\hat{u}\|_{p}^{p}=\lim_{\mu\rightarrow0^{+}}E_{\mu}(\hat{u}_{\mu})=\lim_{\mu\rightarrow0^{+}}\sigma(c, \mu)\leq m(c, 0).
\end{equation*}
Obviously, $m(c, 0)\leq E_{0}(\hat{u})$. Therefore, $E_{0}(\hat{u})=m(c, 0)$ and $\hat{u}$ is a positive solution to \eqref{equ4-11}. Finally, since \eqref{equ4-11} has a unique positive solution $u_{0}$ by Theorem 3.4 of \cite{FF}, $\hat{u}=u_{0}$.

\section{$L^{2}$-supercritical case}

In this section, we always assume that $2+\frac{8s}{3}<q<p<2_{s}^{\ast}$, then $4<q\vartheta_{s, q}<p\vartheta_{s, p}$. We will show Theorems \ref{the1-3}-\ref{the1-4}. Firstly, we can prove that $\mathcal{P}_{c, \mu}^{0}=\emptyset$ and $\mathcal{P}_{c, \mu}$ is a smooth manifold of codimension $2$ in $M_{c}$ in a standard way, see Lemma \ref{lem4-1} or Lemma \ref{lem4-9}.

\begin{lemma}\label{lem5-1}
Let $2+\frac{8s}{3}<q<p<2_{s}^{\ast}$ and $\mu>0$. For every $u\in M_{c}$, $\mathcal{J}_{u}^{\mu}$ has a unique critical point $t_{u}$, which is a strict maximum point at the positive level. Moreover,\\
$(1)$ $\mathcal{P}_{c, \mu}=\mathcal{P}_{c, \mu}^{-}$.\\
$(2)$ $\mathcal{J}_{u}^{\mu}$ is strictly decreasing and concave on $(t_{u}, \infty)$.\\
$(3)$ The map $u\in M_{c}\mapsto t_{u}\in\mathbb{R}$ is of class $C^{1}$.\\
$(4)$ If $P_{\mu}(u)<0$, then $t_{u}<0$.
\end{lemma}

\begin{proof}
From Lemma \ref{lem2-9}, it is easy to see that $\mathcal{J}_{u}^{\mu}$ has a unique maximum point $t_{u}$ at the positive level. By maximality $(\mathcal{J}_{u}^{\mu})''(t_{u})\leq0$, then $(\mathcal{J}_{t_{u}\star u}^{\mu})''(0)=(\mathcal{J}_{u}^{\mu})''(t_{u})\leq0$. Since $t_{u}\star u\in\mathcal{P}_{c, \mu}$ and $\mathcal{P}_{c, \mu}^{0}=\emptyset$, $t_{u}\star u\in\mathcal{P}_{c, \mu}^{-}$. In particular, this with Lemma \ref{lem2-5} imply that $\mathcal{P}_{c, \mu}=\mathcal{P}_{c, \mu}^{-}$. The rest of the proof is very similar to that Lemma 6.1 in \cite{SSSSSSS}, we omit it.

\end{proof}

Next, by adopting a similar argument as the proof in the Lemma \ref{lem4-11} and Lemma \ref{lem4-12}, we have

\begin{lemma}\label{lem5-2}
Let $2+\frac{8s}{3}<q<p<2_{s}^{\ast}$ and $\mu>0$. Then $m(c, \mu):=\inf_{u\in\mathcal{P}_{c, \mu}}E_{\mu}(u)>0$.
\end{lemma}

\begin{lemma}\label{lem5-3}
Let $2+\frac{8s}{3}<q<p<2_{s}^{\ast}$ and $\mu>0$. There exists $k>0$ sufficiently small such that
\begin{equation*}
0<\sup_{\overline{A_{k}}}E_{\mu}<m(c, \mu) \quad and \quad u\in\overline{A_{k}}\Rightarrow E_{\mu}(u), P_{\mu}(u)>0,
\end{equation*}
where $A_{k}=\{u\in M_{c}: \|(-\Delta)^{\frac{s}{2}}u\|_{2}<k\}$.
\end{lemma}

We define the minimax class
\begin{equation*}
\Gamma:=\{\gamma\in C([0, 1], M_{c, r}): \gamma(0)\in\overline{A_{k}}, E_{\mu}(\gamma(1))\leq0\}.
\end{equation*}
Clearly, $\Gamma\neq\emptyset$. In fact, for any $u\in M_{c, r}$, there exist $\tau_{0}\ll-1$ and $\tau_{1}\gg1$ such that $\tau_{0}\star u\in\overline{A_{k}}$, $E_{\mu}(\tau_{1}\star u)<0$ and $\tau\mapsto\tau\star u$ is continuous. Then, we only need let $\gamma(t)=\tau_{t}\star u$.

Next, we can define the minimax value
\begin{equation*}
\sigma(c, \mu)=\inf_{\gamma\in\Gamma}\max_{t\in[0, 1]}E_{\mu}(\gamma(t)).
\end{equation*}

Finally, by applying Lemmas \ref{lem5-1}-\ref{lem5-2} and Proposition \ref{lem3-1}, we can prove Theorems \ref{the1-3}-\ref{the1-4}.

\text{\bf Proof of Theorem \ref{the1-3}.} We divide the proof into five steps.\\
{\bf Step 1.}\ In order to use Lemma \ref{lem2-7}, let us set
\begin{equation*}
\mathcal{F}:=\Gamma, \quad A:=\gamma([0, 1]), \quad F:=\mathcal{P}_{c, \mu}, \quad B:=\overline{A_{k}}\cup E_{\mu}^{0},
\end{equation*}
where $E_{\mu}^{0}:=\{u\in M_{c, r}: E_{\mu}(u)\leq 0\}$. Now, according to Definition \ref{lem2-6}, we show that $\mathcal{F}$ is homotopy-stable family with extended boundary $B$: for any $\gamma\in\Gamma$ and $\eta\in C([0, 1]\times M_{c, r}, M_{c, r})$ satisfying $\eta(t, u)=u$, $(t, u)\in(0\times M_{c, r})\cup([0, 1]\times B)$, we want to get $\eta(1, \gamma(t))\in\Gamma$. In fact, let $\tilde{\gamma}(t)=\eta(1, \gamma(t))$, then $\tilde{\gamma}(0)=\eta(1, \gamma(0))=\gamma(0)\in \mathcal{P}_{c, \mu}^{+}$, $\tilde{\gamma}(1)=\eta(1, \gamma(1))=\gamma(1)\in E_{\mu}^{0}$. So, we get $\eta(1, \gamma(t))\in\Gamma$.

{\bf Step 2.}\ We prove that the condition \eqref{equ2-2} in Lemma \ref{lem2-7} holds.\\
By Lemma \ref{lem5-2} and Lemma \ref{lem5-3}, we can see that $F\cap B=\emptyset$, that is $F\backslash B=F$. Next, we show that
\begin{equation*}
A\cap(F\backslash B)=A\cap F=\gamma([0, 1])\cap\mathcal{P}_{c, \mu}\neq\emptyset.
\end{equation*}
Indeed, since $\gamma\in\Gamma$, then $\gamma(0)\in\overline{A_{k}}$. By Lemma \ref{lem5-3}, we know that $P_{\mu}(\gamma(0))>0$. Moreover, $E_{\mu}(\gamma(1))\leq0$. We consider the fiber map $\mathcal{J}_{\gamma(1)}^{\mu}$, then $t_{\gamma(1)}<0$. By Lemma \ref{lem5-1}-$(2)$, we can see that $P_{\mu}(\gamma(1))=\frac{1}{s}(\mathcal{J}_{\gamma(1)}^{\mu})'(0)<0$. Therefore, by the continuity of $P_{\mu}(\gamma(t))$, we know that for every $\gamma\in\Gamma$, there exists $\tau_{\gamma}\in(0, 1)$ such that $P_{\mu}(\gamma(\tau_{\gamma}))=0$. Hence, we can see that $A\cap(F\backslash B)\neq\emptyset$.

{\bf Step 3.}\ We prove that the condition \eqref{equ2-3} in Lemma \ref{lem2-7} holds.\\
In fact, we need to prove that
\begin{equation*}
\inf_{\mathcal{P}_{c, \mu}}E_{\mu}\geq\sigma(c, \mu)\geq\sup_{\overline{A_{k}}\cup E_{\mu}^{0}}E_{\mu}.
\end{equation*}
By step 2, we know that $\gamma([0, 1])\cap\mathcal{P}_{c, \mu}\neq\emptyset$, then we have
\begin{equation*}
\max_{t\in[0, 1]}E_{\mu}(\gamma(t))\geq\inf_{\mathcal{P}_{c, \mu}}E_{\mu},
\end{equation*}
which implies that $\sigma(c, \mu)\geq m(c, \mu)$. Moreover, if $u\in\mathcal{P}_{c, \mu}$, then for $s_{0}\ll-1$ and $s_{1}\gg1$,
\begin{equation*}
\gamma_{u}: \tau\in[0, 1]\mapsto((1-\tau)s_{0}+\tau s_{1})\star u\in M_{c, r}
\end{equation*}
is a path in $\Gamma$. Since $u\in\mathcal{P}_{c, \mu}$, by Lemma \ref{lem5-1}, we have $t_{u}=0$ is a global maximum point for $\mathcal{J}_{\mu}^{u}$. Hence,
\begin{equation*}
E_{\mu}(u)\geq\max_{t\in[0, 1]}E_{\mu}(\gamma_{u}(t))\geq\sigma(c, \mu),
\end{equation*}
which implies that $m(c, \mu)\geq\sigma(c, \mu)$. So, we have $m(c, \mu)=\sigma(c, \mu)>0$ by Lemma \ref{lem5-2}. In addition, by Lemma \ref{lem5-3}, we can see that $E_{\mu}(u)\leq m(c, \mu)$ for any $u\in\overline{A_{k}}\cup E_{\mu}^{0}$. Hence, we get \eqref{equ2-3} holds.

{\bf Step 4.}\ By Step 1, Step 2 and Step 3, we can see that all the conditions in Lemma \ref{lem2-7} holds. Therefore, by Lemma \ref{lem2-7}, we obtain a Palais-Smale sequence $\{u_{n}\}$ for $E_{\mu}|_{M_{c, r}}$ at the level $\sigma(c, \mu)= m(c, \mu)>0$ and dist$(u_{n}, \mathcal{P}_{c, \mu})\rightarrow0$, i.e. $P_{\mu}(u_{n})\rightarrow0$.

{\bf Step 5.}\ By Step 4 and Proposition \ref{lem3-1}, up to a subsequence $u_{n}\rightarrow\hat{u}_{\mu}$ in $H^{s}(\mathbb{R}^{3})$, with $\hat{u}_{\mu}\in M_{c, r}$ is a nonnegative radial solution to \eqref{equ1-1} for some $\hat{\lambda}_{\mu}<0$. Similar to the arguments of Theorem \ref{the1-1}-$(1)$, we know that $\hat{u}_{\mu}>0$.

\text{\bf Proof of Theorem \ref{the1-4}.} The proof is very similar to that of Theorem \ref{the1-2}-$(2)$, we omit it.

\text{\bf Acknowledgements}
The authors would like to thank the anonymous referees for carefully reading this paper and making valuable comments and suggestions. This work is supported by NSFC grant (12071486).


\begin{thebibliography}{10}










\bibitem{C}
R. Cont, P. Tankov, Financial Modeling with Jump Processes, \emph{Chapman Hall/CRC Financial Mathematics Series, 2004, Boca Raton.}
\bibitem{CC}
S.Y.A. Chang, M. del Mar Gonz$\acute{a}$lez, Fractional Laplacian in conformal geometry, \emph{Adv. Math.}, \textbf{226} (2011), 1410-1432.

\bibitem{CCC}
L. Caffarelli, L. Silvestre, An extension problem related to the fractional Laplacian, \emph{Comm. Partial Differential Equations}, \textbf{32} (2007), 1245-1260.




\bibitem{DD}
E. Di Nezza, G. Palatucci, E. Valdinoci, Hitchhiker's guide to the fractional Sobolev spaces, \emph{Bull. Sci. Math.}, \textbf{136} (2012), 521-573.
\bibitem{DDD}
S. Dipierro, G. Palatucci, E. Valdinoci, Existence and symmetry results for a Schr\"{o}dinger type problem involving the fractional Laplacian, \emph{Matematiche (Catania)}, \textbf{68} (2013), 201-216.









\bibitem{FF}
R. L. Frank, E. Lenzmann, L. Silvestre, Uniqueness of radial solutions for fractional Laplacian, \emph{Comm. Pure Appl. Math.}, \textbf{69} (2016), 1671-1726.


\bibitem{GG}
T. Gou, L. Jeanjean, Multiple positive normalized solutions for nonlinear Schr\"{o}dinger systems, \emph{Nonlinearity}, \textbf{31} (2018), 2319-2345.












\bibitem{HH}X. M. He, W. M. Zou,
Multiplicity of concentrating solutions for a class of fractional Kirchhoff equation, \emph{Manuscripta Math.},
 \textbf{158} (2019), 159-203.

\bibitem{HHH}
Y. He, G. B. Li, Standing waves for a class of Kirchhoff type problems in $\mathbb{R}^{3}$ involving critical Sobolev exponents,
\emph{Calc. Var. Partial Differential Equations}, \textbf{54} (2015), 3067-3106.

\bibitem{HHHH}
X. M. He, W. M. Zou, Existence and concentration behavior of positive solutions for a kirchhoff equation in $\mathbb{R}^{3}$,
\emph{J. Differential Equations}, \textbf{2} (2012), 1813-1834.

\bibitem{I}
N. Ikoma, K. Tanaka, A note on deformation argument for $L^{2}$ normalized solutions of nonlinear
Schr\"{o}dinger equations and systems, \emph{Adv. Differential Equations}, \textbf{24} (2019), 609-646.





\bibitem{J}
L. Jeanjean, Existence of solutions with prescribed norm for semilinear elliptic equations, \emph{Nonlinear Anal.}, \textbf{28} (1997), 1633-1659.





\bibitem{L}
N. Laskin,
Fractional quantum mechanics and L$\acute{e}$vy path integrals, \emph{Phys. Lett. A}, \textbf{268} (2000), 298-305.

\bibitem{LL}
N. Laskin,
Fractional Schr\"{o}dinger equation, \emph{Phys. Rev. E (3)}, \textbf{66} (2002), 56-108.
\bibitem{LLL}
Z. S. Liu, H. J. Luo, Z. T. Zhang, Dancer-Fu$\breve{c}$ik spectrum for fractional Schr\"{o}dinger operators with a steep potential well on $\mathbb{R}^{N}$, \emph{Nonlinear Anal.}, \textbf{189} (2019), 111565.

\bibitem{LLLL}
H. J. Luo, Z. T. Zhang, Normalized solutions to the fractional Schr\"{o}dinger equations
with combined nonlinearities, \emph{Calc. Var. Partial Differential Equations}, \textbf{59} (2020), 1-35.

\bibitem{LLLLL}Z. S. Liu, M. Squassina, J. J. Zhang,
Ground states for fractional Kirchhoff equations with critical nonlinearity in lowdimension,
\emph{NoDEA Nonlinear Differential Equations Appl.}, \textbf{24} (2017), 1-32.


\bibitem{LLLLLL}
G . B. Li, X. Luo, T. Yang, Normalized solutions to a class of Kirchhoff equations with sobolev critical exponent. arXiv:2013.08106v1.


\bibitem{M}
R. Metzler, J. Klafter, The random walks guide to anomalous diffusion: a fractional dynamics approach, \emph{Phys. Rep.}, \textbf{339} (2000), 1-77.




\bibitem{P}
P. Pucci, M. Xiang, B. Zhang, Multiple solutions for nonhomogeneous Schr\"{o}dinger-Kirchhoff type
equations involving the fractional p-Laplacian in $\mathbb{R}^{N}$, \emph{Calc. Var. Partial Differential Equations}, \textbf{54} (2015), 2785-2806.



















\bibitem{S}
L. Silvestre, Regularity of the obstacle problem for a fractional power of the Laplace operator, \emph{Comm. Pure Appl. Math.}, \textbf{60} (2007), 67-112.

\bibitem{SS}
R. Servadei, E. Valdinoci, Variational methods for non-local operators of elliptic type, \emph{Discrete Contin. Dynam. Systems}, \textbf{33} (2013), 2105-2137.




\bibitem{SSSSSS}
N. Soave, Normalized ground states for the NLS equation with combined nonlinearities, \emph{J. Differential Equations}, \textbf{269} (2020), 6941-6987.
\bibitem{SSSSSSS}
N. Soave, Normalized ground states for the NLS equation with combined nonlinearities: The Sobolev critical case, \emph{J. Funct. Anal.}, \textbf{279} (2020), 108610.










\bibitem{WWW}
S. S. Yan, J. F. Yang, X. Yu, Equations involving fractional Laplacian operator: compactness and application, \emph{J. Funct. Anal.}, \textbf{269} (2015), 47-79.






\bibitem{ZZ}
M. D. Zhen, B. L. Zhang, Normalized ground states for the critical fractional NLS equation with a perturbation, \emph{Rev. Mat. Complut.}, (2021).









\end{thebibliography}
\end{document}